\documentclass{amsart}
\usepackage{amsmath,amssymb,amsthm,amsfonts}
\usepackage{hyperref}

\usepackage{color}

\newtheorem{lemma}{Lemma}[section]
\newtheorem{theorem}{Theorem}[section]

\newtheorem{remark}{Remark}[section]
\newtheorem{corollary}{Corollary}[section]

\numberwithin{equation}{section}

\newcommand{\dis}{\displaystyle}

\newcommand{\R}{\mathbb{R}}

\newcommand{\semiG}{\mathbb{A}}

\newcommand{\FM}{\mathbf{M}}
\newcommand{\FP}{\mathbf{P}}
\newcommand{\FL}{\mathbf{L}}
\newcommand{\FI}{\mathbf{I}}

\newcommand{\CD}{\mathcal{D}}
\newcommand{\CE}{\mathcal{E}}
\newcommand{\CF}{\mathcal{F}}
\newcommand{\CH}{\mathcal{H}}

\newcommand{\CN}{\mathcal{N}}
\newcommand{\CZ}{\mathcal{Z}}

\newcommand{\MRk}{\mathfrak{R}}

\newcommand{\na}{\nabla}

\newcommand{\al}{\alpha}
\newcommand{\be}{\beta}
\newcommand{\ga}{\gamma}
\newcommand{\om}{\omega}
\newcommand{\la}{\lambda}
\newcommand{\de}{\delta}
\newcommand{\si}{\sigma}
\newcommand{\pa}{\partial}

\newcommand{\eps}{\epsilon}

\newcommand{\De}{\Delta}
\newcommand{\Ga}{\Gamma}

\newcommand{\lag}{\langle}
\newcommand{\rag}{\rangle}

\begin{document}

\title[Dissipation of the Vlasov-Maxwell-Boltzmann System]{Dissipative property of the
Vlasov-Maxwell-Boltzmann System with a uniform ionic background}
\author{Renjun Duan}
\address{Department of Mathematics, The Chinese University of Hong Kong,
Shatin, Hong Kong} \email{rjduan@math.cuhk.edu.hk}

\date{}


\begin{abstract}
In this paper we discuss the dissipative property of near-equilibrium
classical solutions to the Cauchy problem of the
Vlasov-Maxwell-Boltzmann System in the whole space $\R^3$ when the
positive charged ion flow provides a spatially uniform background. The most key point of studying this coupled degenerately dissipative system here is to establish the dissipation of the electromagnetic field which turns out to be of the regularity-loss type. Precisely,
for the linearized non-homogeneous system, some $L^2$ energy
functionals and $L^2$ time-frequency functionals which are
equivalent with the naturally existing ones are designed to capture
the optimal dissipation rate of the system, which in turn yields
the optimal $L^p$-$L^q$ type time-decay estimates of the
corresponding linearized solution operator. These results show a
special feature of the one-species  Vlasov-Maxwell-Boltzmann
system different from the case of two-species, that is, the
dissipation of the magnetic field in one-species is strictly weaker
than the one in two-species. As a by-product, the
global existence of solutions to the nonlinear Cauchy problem is also proved by
constructing some similar energy functionals  {but the time-decay rates of the obtained solution still remain open}.
\end{abstract}

\maketitle

\thispagestyle{empty}

%{\small {\it Keywords:} { ***}

%\medskip

%AMS Subject Classification (2000): ***}

%\vspace{1cm}

%\setcounter{tocdepth}{1}
\tableofcontents

%\newpage

\section{Introduction}

\subsection{Main results}

The Vlasov-Maxwell-Boltzmann system is an important model for plasma physics to describe the time evolution of dilute charged particles (e.g., electrons and ions in the case of two-species) under the influence of the self-consistent internally generated Lorentz forces \cite{MRS}.  {In physical situations the ion mass is usually much larger than the electron mass so that the electrons move much faster than the ions. Thus, the ions are often
described by a fixed ion background $n_{\rm b}(x)$ and only the electrons move. For such simple case,} the Vlasov-Maxwell-Boltzmann system takes the form of
\begin{equation}\label{VMB.o}
    \left\{\begin{array}{l}
      \dis \pa_t f+\xi\cdot \na_x f+ (E+\xi \times B)\cdot \na_\xi f =Q
      (f,f),\\[3mm]
      \dis\pa_t E-\na_x\times B=-\int_{\R^3}\xi f\, d\xi,\\[3mm]
      \dis\pa_t B +\na_x \times E =0,\\[3mm]
      \dis\na\cdot E=\int_{\R^3}f \,d\xi-n_{\rm b},\ \ \na_x\cdot B=0.
    \end{array}\right.
\end{equation}
Here, the unknowns are $f=f(t,x,\xi): (0,\infty)\times \R^3\times \R^3\to [0,\infty)$, $E=E(t,x): (0,\infty)\times \R^3\to \R^3$ and $B=B(t,x): (0,\infty)\times \R^3\to \R^3$, with $f(t,x,\xi)$ standing for the number
distribution function of one-species of particles (e.g., electrons) which have position
$x=(x_1,x_2,x_3)$ and velocity
$\xi=(\xi_1,\xi_2,\xi_3)$ at time $t$, and $E(t,x)$ and
$B(t,x)$ denoting the electromagnetic field in terms of the time-space variable $(t,x)$. The initial
data of the system at $t=0$ is given by
\begin{equation}\label{VMB.o.ID}
% \nonumber to remove numbering (before each equation)
 f(0,x,\xi)=f_{0}(x,\xi),\ \ E(0,x)=E_0(x),\ \ B(0,x)=B_0(x).
\end{equation}
$Q$ is the bilinear
Boltzmann collision operator \cite{CIP-Book} for the hard-sphere model defined by
\begin{eqnarray*}
% \nonumber to remove numbering (before each equation)
&\dis Q(f,g)=\int_{\R^3\times S^{2}}(f'g_\ast'-fg_\ast)
  |(\xi-\xi_\ast)\cdot\om |d\om d\xi_\ast, \\
&\dis f=f(t,x,\xi),\ \ f'=f(t,x,\xi'), \ \ g_\ast=g(t,x,\xi_\ast), \
\ g_\ast'=g(t,x,\xi_\ast'),\\
&\dis \xi'=\xi-[(\xi-\xi_\ast)\cdot \om]\om,\ \
\xi_\ast'=\xi_\ast+[(\xi-\xi_\ast)\cdot \om]\om,\ \ \om\in S^{2}.
\end{eqnarray*}
 {Notice that system \eqref{VMB.o} in general  contains physical constants such as the charge and mass of electrons and the speed of light, cf.~\cite{GuoVMB}. Since our purpose in this paper is to investigate the dissipative property of solutions near global Maxwellians, those physical constants in   system \eqref{VMB.o} have been normalized to be one for notational simplicity. Through this paper, $n_{\rm b}(x)$ is assumed to be a positive constant denoting the spatially uniform density of the ionic background, and we also set $n_{\rm b}=1$ without loss of generality.}

We are interested in the solution to the Cauchy problem  for the
case when the number distribution function $f(t,x,\xi)$ is near an equilibrium
state $\FM$ and $E(t,x)$ and $B(t,x)$ have small amplitudes, where $\FM$ denotes the normalized Maxwellian
\begin{equation*}
 \FM=\FM(\xi)=(2\pi)^{ -3/2}e^{-|\xi|^2/2}.
\end{equation*}
For that, set the perturbation $u$ by
\begin{equation*}
f(t,x,\xi)=\FM + \FM^{1/2} u(t,x,\xi).
\end{equation*}
Then, the Cauchy problem \eqref{VMB.o}, \eqref{VMB.o.ID} can be
reformulated as
\begin{equation}\label{VMB.eq}
    \left\{\begin{array}{l}
      \dis \pa_t u+\xi\cdot \na_x u+ (E+\xi \times B)\cdot \na_\xi u-\xi\FM^{1/2}\cdot E\\
       \dis \hspace{6cm}=\FL u +\Ga(u,u)+\frac{1}{2}\xi\cdot E u,\\[3mm]
      \dis\pa_t E-\na_x\times B=-\int_{\R^3}\xi \FM^{1/2} u\, d\xi,\\[3mm]
      \dis\pa_t B +\na_x \times E =0,\\[3mm]
      \dis\na\cdot E=\int_{\R^3}\FM^{1/2} u\, d\xi,\ \ \na_x\cdot B=0,
    \end{array}\right.
\end{equation}
with initial data
\begin{eqnarray}\label{VMB.ID}
% \nonumber to remove numbering (before each equation)
&\dis u(0,x,\xi)=u_{0}(x,\xi),\ \ E(0,x)=E_0(x),\ \ B(0,x)=B_0(x).
\end{eqnarray}
Here, the linear term $\FL u$ and the nonlinear term $\Ga(u,u)$ are defined in \eqref{def.L}
and \eqref{def.Ga} later on. The problems to be considered are {\it (i) whether or
not any small amplitude solution
\begin{equation*}
    [u(t),E(t),B(t)]: \R^+\to X=H^N(\R^3_x\times \R^3_\xi)\times
    H^N(\R^3_x)\times H^N(\R^3_x)
\end{equation*}
with a properly large $N$ for the above reformulated Cauchy problem
uniquely exists for all $t>0$ if initial data $[u_0,E_0,B_0]\in X$
is sufficiently small; (ii) if so, does the solution decay in time
with some explicit rate?} We shall give in this paper a satisfactory
answer to the first question and a partial answer to the second one
only in the linearized level. Since these two issues have been extensively studied in different contents such as the Boltzmann equation \cite{Duan,D-Hypo}, the Vlasov-Poisson-Boltzmann system \cite{DY-09VPB,DS-VPB} and the Vlasov-Maxwell-Boltzmann system for two-species \cite{DS-VMB}, our emphasize here will be put on the study of the weakly dissipative property of the electromagnetic field and its resulting slow time-decay rate of solutions, which are  even different from the case of two-species. All details for their discussions are left to the next subsection.

Let us begin with the Cauchy problem on the linearized
non-homogeneous Vlasov-Maxwell-Boltzmann system, in the form of
\begin{equation}\label{VMB.li.eq}
    \left\{\begin{array}{l}
      \dis \pa_t u+\xi\cdot \na_x u-\xi\FM^{1/2}\cdot E
      =\FL u +h,\\[3mm]
      \dis\pa_t E-\na_x\times B=-\int_{\R^3}\xi \FM^{1/2} u\, d\xi,\\[3mm]
      \dis\pa_t B +\na_x \times E =0,\\[3mm]
      \dis\na\cdot E=\int_{\R^3}\FM^{1/2} u\, d\xi,\ \ \na_x\cdot B=0,
    \end{array}\right.
\end{equation}
with
\begin{eqnarray}\label{VMB.li.ID}
% \nonumber to remove numbering (before each equation)
&\dis u(0,x,\xi)=u_{0}(x,\xi),\ \ E(0,x)=E_0(x),\ \ B(0,x)=B_0(x),
\end{eqnarray}
where $h=h(t,x,\xi)$ denotes a given non-homogeneous source term.  For simplicity, we write
\begin{equation*}
    U=[u,E,B],\ \ U_0=[u_0,E_0,B_0].
\end{equation*}
Moreover, $U_0=[u_0,E_0,B_0]$ is always supposed to satisfy the last
equation of \eqref{VMB.li.eq} for $t=0$. The first result, concerning
the naturally existing energy functional and its optimal dissipation
rate, is stated as follows. Here and hereafter, $\nu=\nu(\xi)$ and
$\FP$  are defined in \eqref{def.nu} and \eqref{def.P},
respectively; see Subsection \ref{sec.notation} for more notations used in this
paper.

\begin{theorem}\label{thm.efli}
Let $N\geq 3$. Assume $\nu^{-1/2}h\in L^2_\xi(H^N_x)$ with $\FP h (t,x,\xi)=0$. Define the temporal $L^2$-energy functional $\CE_N^{\rm lin}(U(t))$
and its dissipation rate $\CD_N^{\rm lin}(U(t))$ by
\begin{eqnarray}
    \CE_N^{\rm lin}(U(t))&\sim& \|u(t)\|_{L^2_\xi
    (H^N_x)}^2+\|[E(t),B(t)]\|_{ {H^N_x}}^2,\label{thm.efli.00}\\
    \CD_N^{\rm lin}(U(t)) &= &
    \|\nu^{1/2}\{\FI-\FP\}u(t)\|_{L^2_\xi(H^{N}_x)}^2 +\|\na_x \FP u(t)\|_{L^2_\xi(H^{N-1}_x)}^2\label{thm.efli.01}\\
    &&+\|\na_x E(t)\|_{ {H^{N-2}_x}}^2+\|\na_x^2
    B(t)\|_{ {H^{N-3}_x}}^2.\nonumber
\end{eqnarray}
Then, for any smooth solution  $U=[u,E,B]$ of the Cauchy problem \eqref{VMB.li.eq}-\eqref{VMB.li.ID} belonging to $L^2_\xi(H^N_x)\times  {H^N_x}\times  {H^N_x}$,  there  {exists a continuous functional} $\CE_N^{\rm lin}(U(t))$ given in
\eqref{thm.efli.p7} such
that
\begin{equation}\label{thm.efli.1}
    \frac{d}{dt}\CE_N^{\rm lin}(U(t))+\la\CD_N^{\rm lin}(U(t))\leq C
    \|\nu^{-1/2}h(t)\|_{L^2_\xi(H^N_x)}^2
\end{equation}
for any $t\geq 0$.
\end{theorem}

\begin{remark}\label{rem.phy.diss}
The above theorem shows the precise dissipative
property of the naturally existing $L^2$-energy functional
$$
\|u(t)\|_{L^2_\xi
(H^N_x)}^2+\|[E(t),B(t)]\|_{ {H^N_x}}^2
$$
for a properly large $N$. The construction of the equivalent energy
functional  $\CE_N^{\rm lin}(U(t))$ is used to capture the optimal
dissipation rate $\CD_N^{\rm lin}(U(t))$, which will also be
revisited in Corollary \ref{cor2.tfli}.   This is different from the
case of the two-species Vlasov-Maxwell-Boltzmann system as in
\cite{DS-VMB}, where the dissipation rate is a little stronger due
to a cancelation phenomenon between two species which was firstly
observed in \cite{S.VMB}. Specifically, if it is here supposed that
$h=0$ and $\CE_N^{\rm lin}(U_0)$ is finite, then not only all
macroscopic quantities $\FP u$, $E$, $B$ and the highest-order
derivative $\na_x^N[E,B]$ of the electromagnetic field lose their
time-space integrability, but also the same thing happens to the
second-order derivative $\na_x^2 B$ of the magnetic field. On the
other hand, for the energy space with $m$-order spatial regularity
for any integer $m\geq 0$, its optimal dissipation rate can be
described by $\CD_m^{\rm lin}(U(t))$ once again from Corollary
\ref{cor2.tfli}.  Finally, as seen from Theorem \ref{thm.ls} later
on, this kind of weaker dissipation property leads to some slower
time-decay rates of solutions.
\end{remark}

The second result about some time-frequency functional and its optimal dissipation rate is stated as follows.

\begin{theorem}\label{thm.tfli}
Assume $\nu^{-1/2}\hat{h}\in L^2_\xi$ for $t\geq 0$, $k\in \R^3$,
and $\FP h=0$. Define the $L^2$ time-frequency functional
$\CE(\hat{U}(t,k))$ and its dissipation rate  $\CD(\hat{U}(t,k))$ by
\begin{eqnarray}
% \nonumber to remove numbering (before each equation)
   \CE(\hat{U}(t,k))&\sim& \|\hat{u}\|_{L^2_\xi}^2+|[\hat{E},\hat{B}]|^2,\label{thm.tfli.1}\\
   \CD(\hat{U}(t,k))&=&\|\nu^{1/2}\{\FI-\FP\}\hat{u}\|_{L^2_\xi}+|k\cdot\hat{E}|^2
   +\frac{|k|^2}{1+|k|^2}|[\hat{a},\hat{b},\hat{c}]|^2\label{thm.tfli.2}\\
   &&+\frac{|k|^2}{(1+|k|^2)^2}|\hat{E}|^2+\frac{|k|^4}{(1+|k|^2)^3}|\hat{B}|^2.\nonumber
\end{eqnarray}
Then, for any solution  $U=[u,E,B]$ of the Cauchy problem \eqref{VMB.li.eq} and \eqref{VMB.li.ID} satisfying that $\|\hat{u}\|_{L^2_\xi}^2+|[\hat{E},\hat{B}]|^2$ is finite for $t\geq 0$ and $k\in \R^3$, there indeed exists a time-frequency functional $\CE(\hat{U}(t,k))$ given in \eqref{def.lin.tk} such that
\begin{equation}\label{thm.tfli.3}
    \pa_t\CE(\hat{U}(t,k))+\la\CD(\hat{U}(t,k))\leq C\|\nu^{-1/2}\hat{h}\|_{L^2_\xi}^2
\end{equation}
for any $t\geq 0$ and $k\in \R^3$.
\end{theorem}

\begin{corollary}\label{cor1.tfli}
Under Theorem \ref{thm.tfli},
$\CE(\hat{U}(t,k)) $ further satisfies
\begin{equation}\label{cor1.tfli.1}
    \pa_t \CE(\hat{U}(t,k))+\frac{\la |k|^4}{(1+|k|^2)^3} \CE(\hat{U}(t,k))\leq C\|\nu^{-1/2}\hat{h}\|_{L^2_\xi}^2
\end{equation}
for any $t\geq 0$ and $k\in \R^3$.
\end{corollary}

\begin{remark}
Theorem  \ref{thm.tfli} and Corollary \ref{cor1.tfli} show that the
magnetic field $B$ bears the weakest dissipation property among all
quantities $\{\FI-\FP\}u$, $a$, $b$, $c$, $E$ and $B$, and even the
dissipation of $B$ here is much weaker than that in the case of
two-species Vlasov-Maxwell-Boltzmann system as in \cite{DS-VMB}.
This is also consistent with what has been mentioned in Remark
\ref{rem.phy.diss}.

\end{remark}

\begin{corollary}\label{cor2.tfli}
Let $m\geq 0$ be an integer.  Assume $\nu^{-1/2}\na_x^mh\in L^2_{x,\xi}$ with $\FP h (t,x,\xi)=0$. Define the $L^2$ energy functional $\CE_m^{\rm lin}(U(t))$
and its dissipation rate $\CD_m^{\rm lin}(U(t))$ by
\begin{eqnarray}
    \CE_m^{\rm lin}(U(t))&\sim& \|\na_x^mu(t)\|^2+\|\na_x^m[E(t),B(t)]\|^2,\label{cor2.tfli.1}\\
    \CD_m^{\rm lin}(U(t)) &= &
    \|\nu^{1/2}\na_x^m\{\FI-\FP\}u(t)\|^2 \label{cor2.tfli.2}\\
    &&+\|\na_x^ma\|^2+\|\na_x^{1+m} \lag \na_x \rag^{-1}[a,b,c]\|^2\nonumber\\
    &&+\|\na_x^{1+m} \lag \na_x \rag^{-2}E\|^2
    +\|\na_x^{2+m} \lag \na_x \rag^{-3}B\|^2.\nonumber
\end{eqnarray}
Then, for any smooth solution  $U=[u,E,B]$ of the Cauchy problem \eqref{VMB.li.eq} and \eqref{VMB.li.ID} whose $m$-order spatial derivative belongs to $L^2_{x,\xi}\times L^2\times L^2$,  there indeed exists  {a continuous functional} $\CE_m^{\rm lin}(U(t))$ given in \eqref{def.lin.morder} such
that
\begin{equation}\label{cor2.tfli.3}
    \frac{d}{dt}\CE_m^{\rm lin}(U(t))+\la\CD_m^{\rm lin}(U(t))\leq C
    \|\nu^{-1/2}\na_x^mh(t)\|^2
\end{equation}
for any $t\geq 0$.
\end{corollary}

\begin{remark}
It is straightforward to observe that Corollary  \ref{cor2.tfli}
implies Theorem \ref{thm.efli} by defining
\begin{equation*}
 \CE_N^{\rm lin}(U(t))=\sum_{m=0}^N\CE_m^{\rm lin}(U(t)),\ \ \
 \CD_N^{\rm lin}(U(t))=\sum_{m=0}^N\CD_m^{\rm lin}(U(t)),
\end{equation*}
and  {using}
\begin{equation*}
    \sum_{m=0}^N|\na_x|^{1+m}\lag \na_x\rag^{-1}= |\na_x|\lag
    \na_x\rag^{-1}\sum_{m=0}^N|\na_x|^{m}\sim |\na_x|\lag
    \na_x\rag^{N-1},
\end{equation*}
and likewise
\begin{equation*}
 \sum_{m=0}^N|\na_x|^{1+m}\lag \na_x\rag^{-2}\sim |\na_x|\lag
    \na_x\rag^{N-2},\ \ \ \sum_{m=0}^N|\na_x|^{2+m}\lag \na_x\rag^{-3}\sim |\na_x|^2\lag
    \na_x\rag^{N-3}.
\end{equation*}
 {Notice that the above identities and equivalent relations can be verified with respect to the frequency variable under the Fourier transform.} On the other hand, it is also interesting to see that even when $0\leq m<N$, $ \CD_m^{\rm lin}(U(t))$ can capture the optimal dissipation rate of the naturally existing $m$-order energy functional  $ \CE_m^{\rm lin}(U(t))$.  {For instance, when $m=0$, the direct energy estimate on system \eqref{VMB.li.eq} produces the only dissipation of the microscopic component $\{\FI-\FP\}u$, which is partially contained in the optimal form $\CD_0^{\rm lin}(U(t))$.}
\end{remark}

Furthermore, we can obtain the large-time behavior of solutions to
the linearized non-homogeneous Cauchy problem.
Formally, the solution to the Cauchy problem \eqref{VMB.li.eq} and \eqref{VMB.li.ID}  is denoted
by the summation of two parts,
\begin{eqnarray}
% \nonumber to remove numbering (before each equation)
&& U(t)=U^{I}(t)+U^{II}(t),\label{def.lu}\\
&& U^{I}(t)=\semiG(t)U_0,\ \ U^{I}=[u^I,E^I,B^I], \label{def.lu1}\\
&& U^{II}(t)=\int_0^t\semiG(t-s)[h(s),0,0]ds,\ \
U^{II}=[u^{II},E^{II},B^{II}],\label{def.lu2}
\end{eqnarray}
where $\semiG(t)$ is the linear solution operator for the Cauchy
problem on the linearized homogeneous system corresponding to
\eqref{VMB.li.eq} with $h=0$. Notice that $U^{II}(t)$ is
well-defined because $[h(s),0,0]$ for any $0\leq s\leq t$
satisfies the last equation of \eqref{VMB.li.eq} due to the
fact that $\FP h(s)=0$ and hence
\begin{equation}  \notag
    \int_{\R^3}\FM^{1/2}h(s)d\xi=0.
\end{equation}
For brevity, we introduce the norms
%\marginpar{\Red{INTRODUCE NORMS EARLIER?}}
$\|\cdot\|_{\CH^m}$, $\|\cdot\|_{\CZ_r}$ with $m\geq 0$ and $r\geq
1$ given by
\begin{equation*}
    \|U\|_{\CH^m}^2=\|u\|_{L^2_\xi(H^m_x)}^2+\|[E,B]\|_{H^m_x}^2,\ \  \|U\|_{\CZ_r}=\|u\|_{Z_r}+\|[E,B]\|_{L^r_x},
\end{equation*}
for $U=[u,E,B]$, and note $\CZ_2=\CH^0$.  Then, the third result to describe the time-decay property of the linearized solution is stated as follows.

% {INSERT BRIEF DISCUSSION...}

\begin{theorem}\label{thm.ls}
Let $1\leq p, r\leq 2\leq q\leq \infty$,  $\si\geq 0$, and let $m\geq 0$ be an integer.
Suppose $\FP h=0$. Let
$U$ be defined in \eqref{def.lu}, \eqref{def.lu1} and
\eqref{def.lu2} as the solution to the Cauchy problem \eqref{VMB.li.eq}-\eqref{VMB.li.ID}.
Then, the first part $U^{I}$ corresponding to the solution of the
linearized homogeneous system satisfies
\begin{equation}
 \|\na_x^m U^I(t)\|_{\CZ_q}
\leq C(1+t)^{-\frac{3}{4}(\frac{1}{p}-\frac{1}{q})-\frac{m}{4}}
\|U_0\|_{\CZ_p}
 +C(1+t)^{-\frac{\si}{2}} \|\na_x^{m+[\si+3(\frac{1}{r}-\frac{1}{q})]_+}U_0\|_{\CZ_r},
 \label{thm.ls.1}
\end{equation}
for any $t\geq 0$, and the second part $U^{II}$ corresponding to
the solution of the linearized nonhomogeneous system  with vanishing
initial data satisfies
\begin{multline}
 \|\na_x^m U^{II}(t)\|_{\CZ_2}^2
\leq C\int_0^t (1+t-s)^{-\frac{3}{2}(\frac{1}{r}-\frac{1}{2})-m}\|\nu^{-1/2}h(s)\|_{Z_p}^2ds\\
+C \int_0^t
(1+t-s)^{-\si}\|\nu^{-1/2}\na_x^{m+[\si]_+}h(s)\|^2\,ds,
\label{thm.ls.2}
\end{multline}
for any $t\geq 0$.  {Here, $[\cdot]_+$ is defined  by}
\begin{equation}\label{def.plus} {
    [\si+3(\frac{1}{r}-\frac{1}{q})]_+
    =\left\{\begin{array}{ll}
      \si & \ \ \ \text{if $\si$ is integer and $r=q=2$},\\[3mm]
      {[}\si+3(\frac{1}{r}-\frac{1}{q}){]}+1&\ \ \ \text{otherwise},
     \end{array}\right.}
\end{equation}
 {where $[\cdot]$ means the integer part of the nonnegative argument.}
\end{theorem}

Finally, let us go back to the Cauchy problem on the nonlinear
Vlasov-Maxwell-Boltzmann system. The global existence and uniqueness of solutions are stated as follows.

\begin{theorem}\label{thm.non}
Let $N\geq 4$. Define $L^2$ energy functional $\CE_N(U(t))$ and its
dissipation rate $\CD_N(U(t))$ by
\begin{eqnarray}
    \CE_N(U(t))&\sim& \|u(t)\|_{H^N_{x,\xi}}^2+\|[E(t),B(t)]\|_{ {H^N_x}}^2,\label{thm.non.1}\\
    \CD_N(U(t)) &= &
    \|\nu^{1/2}\{\FI-\FP\}u(t)\|_{H^N_{x,\xi}}^2
    +\|\nu^{1/2}\na_xu(t)\|_{L^2_\xi(H^{N-1}_x)}^2\label{thm.non.2}\\
    &&+\|\na_x E(t)\|_{ {H^{N-2}_x}}^2+\|\na_x^2
    B(t)\|_{ {H^{N-3}_x}}^2.\nonumber
\end{eqnarray}
Suppose $f_0=\FM+\FM^{1/2}u_0\geq 0$. There indeed exists  {a continuous functional}
$\CE_N(U(t))$ given in \eqref{lem.non.p06} or \eqref{lem.non.p07} such that if initial data $U_0=[u_0,E_0,B_0]$ satisfies \eqref{VMB.eq}$_4$ for $t=0$ and  $\CE_N(U_0)$ is sufficiently small, then the nonlinear Cauchy
problem \eqref{VMB.eq}-\eqref{VMB.ID} admits a global solution
$U=[u,E,B]$ satisfying
\begin{eqnarray*}
% \nonumber to remove numbering (before each equation)
&\dis f(t,x,\xi)=\FM+\FM^{1/2}u(t,x,\xi)\geq 0,\\
&\dis [u(t),E(t),B(t)]\in C([0,\infty);H^N_{x,\xi}\times  {H^N_x}\times
 {H^N_x}),
\end{eqnarray*}
and
\begin{equation*}
    \CE_N(U(t))+\la\int_0^t \CD_N(U(s))ds\leq \CE_N(U_0)
\end{equation*}
for any $t\geq 0$.
\end{theorem}

 {The decay rate of the solution obtained in Theorem \ref{thm.non} remains open. We shall discuss at the end of this paper the main difficulty of extending the linear decay property in Theorem \ref{thm.ls} to the time-decay of the nonlinear system.}

\subsection{Related work and key points in the proof}
As mentioned before, the method of constructing energy functionals or time-frequency functionals to deal with the global existence and time-decay estimates presented in this paper has been also extensively applied in \cite{Duan,DY-09VPB,DS-VPB,D-Hypo,DS-VMB}. Specifically, \cite{Duan} is a starting point of these series of work. In \cite{Duan}, some interactive energy functionals were constructed to consider the dissipation of the macroscopic part of the solution and also the global existence of solutions without any initial layer was proved. Later, the same thing was done in \cite{DY-09VPB} for the one-species Vlasov-Poisson-Boltzmann system, where the additional efforts are made to take care of the coupling effect from the self-consistent potential force through the Poisson equation. In order to investigate the optimal rate of convergence for the one-species Vlasov-Poisson-Boltzmann system, a new method on the basis of the linearized Fourier analysis was developed in \cite{DS-VPB} to study the time-decay property of the linear solution operator, where the key point is again to construct some proper time-frequency functionals so as to capture the optimal dissipation rate of the system. At the same time, \cite{D-Hypo} provided another method to study the exponential time-decay for the linear Boltzmann equation with a confining force by using the operator calculations instead of the Fourier analysis.

Recently, following a combination of \cite{DS-VPB} and \cite{S.VMB},
the optimal large-time behavior of the two-species
Vlasov-Maxwell-Boltzmann system was analyzed in \cite{DS-VMB}. The
main finding in \cite{DS-VMB} is that although the non-homogeneous
Maxwell system conserves the energy of the electromagnetic field,
the coupling of the Boltzmann equation with the Maxwell system can
generate some weak dissipation of the electromagnetic field which is
actually of the regularity-loss type. It should be pointed out that
even though the form of two-species Vlasov-Maxwell-Boltzmann system
looks more complicated than that of the case of one-species, the
study of global existence and time-decay rate is much more delicate
in the case of one-species because  the coupling term in the source
of the Maxwell system
\begin{equation*}
   -\int_{\R^3}\xi \FM^{1/2} u(t,x,\xi) d\xi=-b(t,x)
\end{equation*}
corresponds to the momentum component of the macroscopic part of the
solution which is degenerate with respect to the linearized operator
$\FL$. Essentially, it is  this kind of the macroscopic coupling
feature that leads to some different dissipation properties between
two-species and one-species for the Vlasov-Maxwell-Boltzmann system.

For the convenience of readers, let us list a table below to present in a clear way similarity and
difference of the dissipative and time-decay properties for three models: Boltzmann equation (BE), Vlasov-Poisson-Boltzmann system (VPB) and Vlasov-Maxwell-Boltzmann system (VMB).
In Table 1, 1-s means one-species and 2-s two-species. Corresponding to different models, $\CE(t,k)$ stands for some time-frequency functional equivalent with the naturally existing one and $\CD(t,k)$ denotes the optimal dissipation rate of $\CE(t,k)$ satisfying
\begin{equation*}
    \frac{d}{dt}\CE(t,k)+\la \CD(t,k)\leq 0
\end{equation*}
for all $t\geq 0$ and $k\in \R^3$. All estimates are written for the linearized homogeneous equation or system. For more details and proof of other models in the above table, interested readers can refer to \cite{D-Hypo,DS-VPB,DS-VMB}. {Notice that the 1-s VMB system decays faster than the 1-s VPB system due to the choice of initial data, that is, $E_0\in L^1_x$ is assumed in the case of 1-s VMB system, whereas in the case of 1-s VPB system, the electric field $E_0=-\na_x\phi_0$ with the potential force $\phi_0$ satisfying the Poisson equation $-\De_x\phi_0=\int_{\R^3}\FM^{1/2} u_0\,d\xi$ may not belong to $L^1_x$ under the assumption of $u_0\in Z_1\cap L^2$. We remark that the decay rate $(1+t)^{-1/4}$ for the 1-s VPB system can be improved to be $(1+t)^{-3/4}$ provided that $\na_x\phi_0\in L^1_x$ is additionally supposed.  Finally, it should also} be pointed out that the method developed in \cite{D-Hypo,DS-VPB,DS-VMB} and this paper could provide a good tool to deal with similar studies for other physical models with the structure involving not only the free transport operator but also the degenerately dissipative operator; see also \cite{Vi}.

\begin{table}[h]
\centering
\begin{tabular}{llll}
\hline &\vline\
$\CE(t,k)\sim$
&\vline\ $\CD(t,k)=$
&\vline\ $\|u(t)\|_{L^2}\leq $\\
\hline BE &\vline\  $\|\hat{u}\|_{L^2_\xi}^2$ &\vline \
$\begin{array}{l}
\\[-2mm]
\|\nu^{1/2}\{\FI-\FP\} {\hat{u}}\|_{L^2_\xi}^2\\
+\frac{|k|^2}{1+|k|^2}|[\hat{a},\hat{b},\hat{c}]|^2\\[-2mm]
\ \
\end{array}$
&\vline\ $C(1+t)^{-\frac{3}{4}}\|u_0\|_{Z_1\cap L^2}$ \\
\hline 1-s VPB &\vline\
$\|\hat{u}\|_{L^2_\xi}^2+\frac{|\hat{a}|^2}{|k|^2}$     &\vline\
$\begin{array}{l}
\\[-2mm]
\|\nu^{1/2}\{\FI-\FP\} {\hat{u}}\|_{L^2_\xi}^2\\
+\frac{|k|^2}{1+|k|^2}|[\hat{a},\hat{b},\hat{c}]|^2+|\hat{a}|^2\\[-2mm]
\ \ \
\end{array}$
&\vline\ $C(1+t)^{-\frac{1}{4}}\|u_0\|_{Z_1\cap L^2}$
\\
\hline 1-s VMB    &\vline\
$\|\hat{u}\|_{L^2_\xi}^2+|[\hat{E},\hat{B}]|^2$ &\vline\
$\begin{array}{l}
\\[-2mm]
   \|\nu^{1/2}\{\FI-\FP\}\hat{u}\|_{L^2_\xi}\\
   +\frac{|k|^2}{1+|k|^2}|[\hat{a},\hat{b},\hat{c}]|^2+|k\cdot\hat{E}|^2\\
   +\frac{|k|^2}{(1+|k|^2)^2}|\hat{E}|^2+\frac{|k|^4}{(1+|k|^2)^3}|\hat{B}|^2\\[-2mm]
   \ \
 \end{array}$
&\vline\ $\begin{array}{l}
           C(1+t)^{-\frac{3}{8}}\\
           \times (\|U_0\|_{\CZ_1}+\|\na_xU_0\|_{\CZ_2})
         \end{array}$
\\\hline 2-s VMB    &\vline\
$\|\hat{u}\|_{L^2_\xi}^2+|[\hat{E},\hat{B}]|^2$ &\vline\
$\begin{array}{l}
\\[-2mm]
   \|\nu^{1/2}\{\FI-\FP\}\hat{u}\|_{L^2_\xi}\\
   +\frac{|k|^2}{1+|k|^2}|[\hat{a}_\pm,\hat{b},\hat{c}]|^2+|k\cdot\hat{E}|^2\\
   +\frac{|k|^2}{(1+|k|^2)^2}(|\hat{E}|^2+|\hat{B}|^2)\\[-2mm]
   \ \
 \end{array}$
&\vline\ $\begin{array}{l}
           C(1+t)^{-\frac{3}{4}}\\
           \times (\|U_0\|_{\CZ_1}+\|\na_x^2U_0\|_{\CZ_2})
         \end{array}$ \\
  \hline\\
\end{tabular}
\caption{Dissipative and time-decay properties of different models}
\end{table}

Since the current work is a further development in the study of the Vlasov-Maxwell-Boltzmann system as in \cite{DS-VMB}, we omit the detailed literature review for brevity, and readers can refer to \cite{DS-VMB} and reference therein. Here, we only mention some of them. The spectral analysis and global existence for the Boltzmann equation with near-equilibrium initial data was given by \cite{Ukai-1974}. For the same topic, thirteen moments method and global existence was found by \cite{Ka-BE13}. The energy method of the Boltzmann equation was developed independently in \cite{Guo-IUMJ,Guo2,GuoVMB} and \cite{Liu-Yu-Shock,Liu-Yang-Yu, YZ-CMP} by using the different macro-micro decomposition. The almost exponential rate of convergence of the Boltzmann equation on torus for large initial data was obtained in \cite{DV} under some additional regularity assumption. \cite{SG}
provided a very simple proof of \cite{DV} in the framework of small perturbation. The diffusive limit of the two-species Vlasov-Maxwell-Boltzmann system over the torus was studied in \cite{Jang}. It could be interesting to consider the same issue as in \cite{Jang} for the one-species Vlasov-Maxwell-Boltzmann system because of its weaker dissipation property.

In what follows, let us explain some new technical points in the proof of our main results which are different from previous work. The first one is about the dissipation estimate on the momentum component $b(t,x)$ in the macroscopic part $\FP u$ in Theorem \ref{thm.efli}. Recall that in the previous work \cite{DY-09VPB} and \cite{D-Hypo}, the dissipation estimate of $b(t,x)$ was based on \eqref{ME.h}$_1$ and \eqref{ME.bl}$_2$. This fails in the case of one-species Vlasov-Maxwell-Boltzmann system because it is impossible to control a term such as
\begin{equation*}
    \sum_{ij}\de_{ij}\int c(\pa_iE_j+\pa_j E_i)dx.
\end{equation*}
Instead, the right way is to make estimates on \eqref{ME.eq.macro}$_4$. Therefore, one can get the macroscopic dissipation estimate \eqref{thm.efli.p3} with the dissipation of $E$ multiplied by a small constant on the right-hand side. The second key point is about the dissipation of the electromagnetic field. Different from \cite{DS-VMB}, although there is no cancelation in the case of one-species, one can still design some interactive energy functional $\CE_N^{{\rm lin},2}(U(t))$ to capture the weaker dissipation rate
\begin{equation*}
    \sum_{1\leq |\al|\leq N-1}\|\pa^\al E\|^2+ \sum_{2\leq |\al|\leq N-1}\|\pa^\al B\|^2.
\end{equation*}
The third key point is about the time-decay estimate on the linearized solution operator. In fact, for a general frequency function $\phi(k)$ which is of the regularity-loss type as in Lemma \ref{lem.glidecay}, one can repeatedly apply the Minkowski inequality to make interchanges between frequency and space variables so that the more general $L^p$-$L^q$ type time-decay than in Theorem \ref{thm.ls} can be obtained.  The last key point is about the estimate on the nonlinear term $\xi\cdot E u$ in the proof of Lemma \ref{lem.non} concerning the a priori estimates of solutions. {In particular, it is impossible to bound $\iint \xi \cdot E (\FP u)^2\,dxd\xi$ by using $\CE_N(U(t))^{1/2}\CD_N(U(t))$ up to a constant since both $E$ and $\FP u$ do not enter into the dissipation rate $\CD_N(U(t))$ given in \eqref{thm.non.2}. Instead, we first take the velocity integration and then use the macroscopic balance laws \eqref{ME.bl} so as to obtain an estimate as}
$$
 {\iint \xi \cdot E (\FP u)^2\,dxd\xi\leq \frac{d}{dt}\int |b|^2(a+2c)\,dx + C \left[\CE_N(U(t))^{1/2}+\CE_N(U(t))\right]\CD_N(U(t)).}
$$
We remark that this also has been observed  {in \cite{DY-09VPB}} in the study of the one-species Vlasov-Poisson-Boltzmann system.

\subsection{Notations}\label{sec.notation}

Throughout this paper,  $C$  denotes
some positive (generally large) constant and $\la$ denotes some positive (generally small) constant, where both $C$ and
$\la$ may take different values in different places. In addition,
$A\sim B$ means $\la A\leq B \leq \frac{1}{\la} A$ for a generic
constant $0<\la<1$. For any
integer $m\geq 0$, we use $H^m_{x,\xi}$, $H^m_x$, $H^m_\xi$ to denote the
usual Hilbert spaces $H^m(\R^3_x\times\R^3_\xi)$, $H^m(\R^3_x)$,
$H^m(\R^3_\xi)$, respectively, and $L^2$, $L^2_x$, $L^2_\xi$ are
used for the case when $m=0$. When without confusion, we use $H^m$ to denote $H^m_x$ and use $L^2$ to denote $L^2_x$
 or $L^2_{x,\xi}$. For a Banach space $X$,
$\|\cdot\|_{X}$ denotes the corresponding norm, while $\|\cdot\|$
always denotes the norm $\|\cdot\|_{L^2}$ for
simplicity.
For $r\geq 1$, we also define the standard time-space mixed Lebesgue
space $Z_r=L^2_\xi(L^r_x)=L^2(\R^3_\xi;L^r(\R^3_x))$ with the norm
\begin{equation*}
\|g\|_{Z_r}=\left(\int_{\R^3}\left(\int_{\R^3}
    |g(x,\xi)|^rdx\right)^{2/r}d\xi\right)^{1/2},\ \ g=g(x,\xi)\in Z_r.
\end{equation*}
For
multi-indices $\al=[\al_1,\al_2,\al_3]$ and
$\be=[\be_1,\be_2,\be_3]$, we denote
\begin{equation*}
 \pa^{\al}_\be=\pa_{x_1}^{\al_1}\pa_{x_2}^{\al_2}\pa_{x_3}^{\al_3}
    \pa_{\xi_1}^{\be_1}\pa_{\xi_2}^{\be_2}\pa_{\xi_3}^{\be_3}.
\end{equation*}
The length of $\al$ is $|\al|=\al_1+\al_2+\al_3$ and the length of
$\be$ is $|\be|=\be_1+\be_2+\be_3$. For simplicity,
we also use $\pa_j$ to denote $\pa_{x_j}$ for each $j=1,2,3$. For an integrable function $g: \R^3\to\R$, its Fourier transform
%$\widehat{g}=\CF g$
is defined by
\begin{equation*}
  \widehat{g}(k)= \CF g(k)= \int_{\R^3} e^{-2\pi i x\cdot k} g(x)dx, \quad
  x\cdot
   k:=\sum_{j=1}^3 x_jk_j,
   \quad
   k\in\R^3,
\end{equation*}
where $i =\sqrt{-1}\in \mathbb{C}$ is the imaginary unit. For two
complex vectors $a,b\in\mathbb{C}^3$, $(a\mid b)$ denotes the dot product of $a$ with the complex conjugate of $b$ over the complex field. $\lag \na_x\rag=(1+|\na_x|^2)^{1/2}$ is defined in terms of the Fourier transform.

The rest of this paper is arranged as follows. In Section \ref{sec.moment} we present some basic property of the linearized collision operator and derive some macroscopic moment equations. In Section \ref{sec.linear} we study the linearized non-homogeneous Vlasov-Maxwell-Boltzmann system in order to prove Theorem \ref{thm.efli}, Theorem \ref{thm.tfli} and Theorem \ref{thm.ls}. Finally, we prove in Section \ref{sec.nonlinear} Theorem \ref{thm.non} for the global existence of solutions to the nonlinear Cauchy problem.

\section{Moment equations}\label{sec.moment}

It is easy to see that  $\FL u$ and $\Ga(u,u)$ are given by
\begin{eqnarray}
% \nonumber to remove numbering (before each equation)
  \FL u &=& \frac{1}{\sqrt{\FM}}\left[Q(\FM,\sqrt{\FM}u)+Q(\sqrt{\FM}u,\FM)\right],\label{def.L} \\
  \Ga(u,u)&=&
  \frac{1}{\sqrt{\FM}}Q(\sqrt{\FM}u,\sqrt{\FM}u).\label{def.Ga}
\end{eqnarray}
For the linearized collision operator $\FL$, one has the following standard facts \cite{CIP-Book}. $\FL$ can be split as $\FL u=-\nu(\xi) u+K u$, where
 the collision frequency is given by
\begin{equation}\label{def.nu}
\nu(\xi)=\iint_{\R^3\times
  S^{2}}|(\xi-\xi_\ast)\cdot\om|\FM(\xi_\ast)\, d\om
  d\xi_\ast.
\end{equation}
 Notice that $\nu(\xi)\sim (1+|\xi|^2)^{1/2}$.  The null space of $\FL$ is given by
\begin{equation}\notag
    \CN={\rm span}\left\{\FM^{1/2},\xi_i\FM^{1/2}~(1\leq i\leq 3),
    |\xi|^2\FM^{1/2} \right\}.
\end{equation}
The linearized collision operator $\FL$ is non-positive and further $-\FL$ is known to be locally coercive in
the sense that there is a constant $\la_0>0$ such that \cite{CIP-Book}:
\begin{equation}\label{coerc}
    -\int_{\R^3}u\FL u\,d\xi\geq \la_0\int_{\R^3}\nu(\xi)|\{\FI-\FP\}u|^2d\xi,
\end{equation}
where, for fixed $(t,x)$, $\FP$ denotes the orthogonal projection from
$L^2_\xi$ to $\CN$. Given any $u(t,x,\xi)$, one can write $\FP$ in \eqref{coerc} as
\begin{equation}\label{def.P}
    \FP u= \{a(t,x)+b(t,x)\cdot \xi+c(t,x)(|\xi|^2-3)\}\FM^{1/2}.
\end{equation}
Since $\FP$ is a projection, the coefficient functions $a(t,x)$, $b(t,x)\equiv [b_1(t,x),b_2(t,x),b_3(t,x)]$ and $c(t,x)$ depend on $u(t,x,\xi)$ in terms of
\begin{equation}
\notag
\left\{\begin{split}
  \dis   & a= \int_{\R^3} \FM^{1/2} ud\xi= \int_{\R^3} \FM^{1/2} \FP ud\xi,
  \\
  \dis  & b_i=\int_{\R^3}\xi_i \FM^{1/2}ud\xi
=\int_{\R^3} \xi_i \FM^{1/2}\FP ud\xi,\ \ 1\leq i\leq 3,
\\
  \dis & c= \frac{1}{6}\int_{\R^3}(|\xi|^2-3) \FM^{1/2} ud\xi
= \frac{1}{6}\int_{\R^3} (|\xi|^2-3) \FM^{1/2} \FP ud\xi.
    \end{split}\right.
\end{equation}

To derive evolution equations of $a$, $b$ and $c$, we start from the local balance laws of the original system \eqref{VMB.o} to obtain
\begin{equation*}
    \left\{\begin{split}
    &\dis \pa_t \int_{\R^3} f d\xi+\na_x\cdot \int_{\R^3} \xi fd\xi=0,\\
    &\dis \pa_t \int_{\R^3} \xi f d\xi +\na_x \cdot \int_{\R^3} \xi \otimes \xi f d\xi -E \int_{\R^3} f d\xi -\int_{\R^3} \xi f d\xi \times B=0,\\
    &\dis \pa_t \int_{\R^3} \frac{1}{2}|\xi|^2 f d\xi+\na_x \cdot \int_{\R^3} \frac{1}{2}|\xi|^2 \xi f d\xi-E\cdot \int_{\R^3} \xi f d\xi=0.
      \end{split}\right.
\end{equation*}
The above system implies
\begin{equation}\label{ME.bl}
    \left\{\begin{split}
    &\dis \pa_t a+\na_x\cdot b=0,\\
    &\dis \pa_t b +\na_x (a+2c)+\na_x \cdot \int_{\R^3} \xi \otimes \xi \FM^{1/2} \{\FI-\FP\}u d\xi
    -E (1+a) -b \times B=0,\\
    &\dis \pa_t c+\frac{1}{3}\na_x\cdot b+\frac{1}{6}\na_x \cdot \int_{\R^3} |\xi|^2 \xi\FM^{1/2} \{\FI-\FP\}u dx -\frac{1}{3}E\cdot b=0.
      \end{split}\right.
\end{equation}
As in \cite{DS-VPB}, define
\begin{equation*}
    \Theta_{ij}(u)=\int_{\R^3}(\xi_i\xi_j -1)\FM^{1/2} ud\xi,\ \ \Lambda_i(u)=\frac{1}{10}\int_{\R^3} (|\xi|^2-5)\xi_i\FM^{1/2} ud\xi
\end{equation*}
for $1\leq i,j\leq 3$. Applying them to the first equation of \eqref{VMB.eq}, one has
\begin{equation}\label{ME.h}
% \nonumber to remove numbering (before each equation)
\left\{\begin{split}
&\pa_t [\Theta_{ij}(\{\FI-\FP\}u)+2c\de_{ij}]+\pa_ib_j+\pa_j b_i=\Theta_{ij}(\ell+g),\\
&\pa_t \Lambda_i(\{\FI-\FP\}u)+\pa_i  c=\Lambda_i(\ell+g),
\end{split}\right.
\end{equation}
where $\de_{ij}$ means Kronecker delta, and
\begin{equation}\label{def.ell.g}
% \nonumber to remove numbering (before each equation)
 \left\{\begin{split}
  \ell &= -\xi\cdot \na_x\cdot  \{\FI-\FP\}u+\FL u,\\
  g &= \frac{1}{2}\xi\cdot E u-(E+\xi\times B)\cdot \na_\xi u+\Ga(u,u).
  \end{split}\right.
\end{equation}
One can replace $\pa_t c$ in \eqref{ME.h}$_1$ by using \eqref{ME.bl}$_3$ so that
\begin{multline*}
\pa_t \Theta_{ij}(\{\FI-\FP\}u)+\pa_i b_j+\pa_j b_i-\frac{2}{3}\de_{ij}\na_x\cdot b\\
-\frac{10}{3}\de_{ij}\na_x\cdot \Lambda(\{\FI-\FP\}u)
=\Theta_{ij}(\ell+g)
-\frac{2}{3}\de_{ij}E\cdot b.
\end{multline*}
In a summary, we obtained the following moment system
\begin{equation}\label{ME.eq.macro}
\left\{\begin{split}
&\dis \pa_t a+\na_x\cdot b=0,\\
&\dis \pa_t b +\na_x (a+2c)+\na_x \cdot\Theta(\{\FI-\FP\}u)
    -E =Ea+b \times B,\\
&\dis \pa_t c+\frac{1}{3}\na_x\cdot b+\frac{5}{3}\na_x \cdot  \Lambda(\{\FI-\FP\}u)=\frac{1}{3}E\cdot b,\\
&\dis \pa_t \Theta_{ij}(\{\FI-\FP\}u)+\pa_i b_j+\pa_j b_i-\frac{2}{3}\de_{ij}\na_x\cdot b-\frac{10}{3}\de_{ij}\na_x\cdot \Lambda(\{\FI-\FP\}u)\\
&\dis\hspace{2cm}=\Theta_{ij}(\ell+g)
-\frac{2}{3}\de_{ij}E\cdot b,\ 1\leq i,j\leq 3,\\
&\pa_t \Lambda_i(\{\FI-\FP\}u)+\pa_i  c=\Lambda_i(\ell+g),\ 1\leq i\leq 3.
  \end{split}\right.
\end{equation}
On the other hand, the Maxwell system in \eqref{VMB.ID} is equivalent with
\begin{equation*}
     \left\{\begin{split}
   &\dis\pa_t E-\na_x\times B=-b,\\
   &\dis\pa_t B+\na_x\times E =0,\\
   &\dis\na_x\cdot E=0,\ \ \na_x\cdot B=0.
  \end{split}\right.
\end{equation*}
Finally, we should point out that the key analysis of all results in this paper is based on the above moment equations coupled with the Maxwell system.

\section{Linear non-homogeneous system}\label{sec.linear}

In this section we consider the Cauchy problem \eqref{VMB.li.eq} and \eqref{VMB.li.ID} on the linearized Vlasov-Maxwell-Boltzmann system. For convenience of readers, recall it by
\begin{equation}\label{li.eq}
    \left\{\begin{array}{l}
      \dis \pa_t u+\xi\cdot \na_x u-\xi\FM^{1/2}\cdot E
      =\FL u +h,\\[3mm]
      \dis\pa_t E-\na_x\times B=-\int_{\R^3}\xi \FM^{1/2} u d\xi,\\[3mm]
      \dis\pa_t B +\na_x \times E =0,\\[3mm]
      \dis\na\cdot E=\int_{\R^3}\FM^{1/2} u d\xi,\ \ \na_x\cdot B=0,
    \end{array}\right.
\end{equation}
with
\begin{eqnarray}\label{li.ID}
% \nonumber to remove numbering (before each equation)
&\dis u(0,x,\xi)=u_{0}(x,\xi),\ \ E(0,x)=E_0(x),\ \ B(0,x)=B_0(x).
\end{eqnarray}
Here, $h=h(t,x,\xi)$ is a given non-homogenous source term, satisfying $\FP h=0$.

\subsection{$L^2$ energy functional and its optimal dissipation rate}\label{sec.ef}
In this subsection we shall prove Theorem \ref{thm.efli}. Before that,
similar to obtain \eqref{ME.eq.macro}, one can also derive the following moment equations corresponding to the linear equation \eqref{li.eq}:
\begin{equation}\label{ME.lieq.macro}
\left\{\begin{split}
&\dis \pa_t a+\na_x\cdot b=0,\\
&\dis \pa_t b +\na_x (a+2c)+\na_x \cdot\Theta(\{\FI-\FP\}u)
    -E =0,\\
&\dis \pa_t c+\frac{1}{3}\na_x\cdot b+\frac{5}{3}\na_x \cdot  \Lambda(\{\FI-\FP\}u)=0,\\
&\dis \pa_t \Theta_{ij}(\{\FI-\FP\}u)+\pa_i b_j+\pa_j b_i-\frac{2}{3}\de_{ij}\na_x\cdot b-\frac{10}{3}\de_{ij}\na_x\cdot \Lambda(\{\FI-\FP\}u)=\Theta_{ij}(\ell+h),\\
&\pa_t \Lambda_i(\{\FI-\FP\}u)+\pa_i  c=\Lambda_i(\ell+h),
  \end{split}\right.
\end{equation}
where $1\leq i,j\leq 3$, and as in \eqref{def.ell.g}, $\ell$ still denotes
\begin{equation*}
     \ell = -\xi\cdot \na_x\cdot  \{\FI-\FP\}u+\FL u.
\end{equation*}
The Maxwell system also takes the form of
\begin{equation}\label{ME.lieq.maxwell}
     \left\{\begin{split}
   &\dis\pa_t E-\na_x\times B=-b,\\
   &\dis\pa_t B+\na_x\times E =0,\\
   &\dis\na_x\cdot E=0,\ \ \na_x\cdot B=0.
  \end{split}\right.
\end{equation}

\medskip

\noindent{\bf Proof of Theorem \ref{thm.efli}:} Let $N\geq 3$. First of all, a usual energy estimate on \eqref{li.eq} gives
\begin{equation}\label{thm.efli.p1}
    \frac{1}{2}\sum_{|\al|\leq N}\frac{d}{dt}(\|\pa^\al u\|^2+\|\pa^\al [E,B]\|^2)
+\la \sum_{|\al|\leq N}\|\nu^{1/2}\pa^\al \{\FI-\FP\}u\|^2\leq
    C\sum_{|\al|\leq N}\|\nu^{-1/2}\pa^\al g\|^2.
\end{equation}
As in \cite{D-Hypo,DY-09VPB} or \cite{DS-VMB}, one can further deduce from \eqref{ME.lieq.macro} the dissipation of $a$, $b$ and $c$. In fact,
let $\eps_1>0$, $\eps_2>0$ be arbitrary constants to be chosen later. From \eqref{ME.lieq.macro}$_5$ and \eqref{ME.lieq.macro}$_3$, it follows that
\begin{multline*}
\frac{d}{dt}\sum_{|\al|\leq N-1}\int_{\R^3}\na_x\pa^\al c\cdot \Lambda(\pa^\al\{\FI-\FP\}u)dx
+\la \sum_{|\al|\leq N-1}\|\na_x\pa^\al c\|^2\\
\leq \eps_1 \sum_{|\al|\leq N-1}\|\na_x \pa^\al b\|^2
+\frac{C}{\eps_1}\left(\sum_{|\al|\leq N}
\|\pa^\al\{\FI-\FP\}u\|^2+\sum_{|\al|\leq N-1}\|\nu^{-1/2}\pa^\al g\|^2\right).
\end{multline*}
\eqref{ME.lieq.macro}$_4$ and \eqref{ME.lieq.macro}$_2$ imply
\begin{multline*}
\frac{d}{dt}\sum_{|\al|\leq N-1}\sum_{ij=1}^3\int_{\R^3}
(\pa_i\pa^\al b_j+\pa_j\pa^\al b_i-\frac{2}{3}\de_{ij}\na_x\cdot \pa^\al b)
\Theta_{ij}(\pa^\al\{\FI-\FP\}u)dx\\
+\la \sum_{|\al|\leq N-1}\|\na_x \pa^\al b\|^2
\leq \eps_2\left(\sum_{|\al|\leq N-1}\|\na_x\pa^\al [a,c]\|^2+\sum_{|\al|\leq N-2}\|\na_x\pa^\al E\|^2\right)\\
+\frac{C}{\eps_2}\left(\sum_{|\al|\leq N}
\|\pa^\al\{\FI-\FP\}u\|^2+\sum_{|\al|\leq N-1}\|\nu^{-1/2}\pa^\al g\|^2\right).
\end{multline*}
It holds from \eqref{ME.lieq.macro}$_2$ and \eqref{ME.lieq.macro}$_1$ that
\begin{multline*}
\frac{d}{dt}\sum_{|\al|\leq N-1}\int_{\R^3}\na_x\pa^\al a\cdot \pa^\al b dx +\la \sum_{|\al|\leq N}\|\pa^\al a\|^2\\
\leq C \sum_{|\al|\leq N-1}\|\na_x\pa^\al [b,c]\|^2 +C\sum_{|\al|\leq N-1}
\|\na_x\pa^\al\{\FI-\FP\}u\|^2.
\end{multline*}
Define
\begin{equation}
\begin{split}\label{thm.efli.p1-def}
  \CE^{{\rm lin},1}_N(U(t)) =&\sum_{|\al|\leq N-1}\int_{\R^3}\na_x\pa^\al c\cdot \Lambda(\pa^\al\{\FI-\FP\}u)dx \\
  &+\sum_{|\al|\leq N-1}\sum_{ij=1}^3\int_{\R^3}
(\pa_i\pa^\al b_j+\pa_j\pa^\al b_i-\frac{2}{3}\de_{ij}\na_x\cdot \pa^\al b)
\Theta_{ij}(\pa^\al\{\FI-\FP\}u)dx\\
&+\kappa_1\sum_{|\al|\leq N-1}\int_{\R^3}\na_x\pa^\al a\cdot \pa^\al b dx,
\end{split}
\end{equation}
for some small constant $\kappa_1>0$. Therefore, by taking $\kappa_1>0$ small enough and then letting $\eps_1>0$ and $\eps_2>0$ be small enough, the dissipation of $a$, $b$ and $c$ can be obtained by the following inequality
\begin{multline}\label{thm.efli.p2}
\frac{d}{dt}\CE^{{\rm lin},1}_N(U(t))+\la \sum_{|\al|\leq N-1}\|\na_x\pa^\al [a,b,c]\|^2
+\|a\|^2\\
\leq \eps_2\sum_{|\al|\leq N-2}\|\na_x\pa^\al E\|^2
+\frac{C}{\eps_2} \left(\sum_{|\al|\leq N}
\|\pa^\al\{\FI-\FP\}u\|^2+\sum_{|\al|\leq N-1}\|\nu^{-1/2}\pa^\al g\|^2\right),
\end{multline}
where $\eps_2>0$ is still left to be chosen later on.

The key part is to estimate the dissipation of $E$ and $B$. We claim that
\begin{multline}\label{thm.efli.p3}
-\frac{d}{dt}\sum_{1\leq |\al|\leq N-1}\int_{\R^3}\pa^\al E\cdot \pa^\al b dx
+\la \sum_{1\leq |\al|\leq N-1}\|\pa^\al E\|^2\\
\leq \eps_3 \sum_{1\leq |\al|\leq N-2}\|\pa^\al \na_x\times B\|^2+\frac{C}{\eps_3}\sum_{|\al|\leq N-1}\|\na_x\pa^\al b\|^2\\
+C \sum_{1\leq |\al|\leq N-1}\left(\|\na_x\pa^\al [a,c]\|^2+\|\na_x \pa^\al \{\FI-\FP\}u\|^2\right),
\end{multline}
with $\eps_3>0$ to be chosen, and
\begin{multline}\label{thm.efli.p4}
-\frac{d}{dt}\sum_{1\leq |\al|\leq N-2}\int_{\R^3}\pa^\al \na_x\times B\cdot \pa^\al E dx +\la \sum_{1\leq |\al|\leq N-2}\|\pa^\al \na_x\times B\|^2\\
\leq \sum_{1\leq |\al|\leq N-2}\|\pa^\al \na_x\times E\|^2
+C \sum_{1\leq |\al|\leq N-2}\|\pa^\al b\|^2.
\end{multline}
For this time, suppose that \eqref{thm.efli.p3} and \eqref{thm.efli.p4} hold true. Define
\begin{equation}\label{thm.efli.p5}
% \nonumber to remove numbering (before each equation)
  \CE^{{\rm lin},2}_N(U(t)) =-\sum_{1\leq |\al|\leq N-1}\int_{\R^3}\pa^\al E\cdot \pa^\al b dx
  -\kappa_2\sum_{1\leq |\al|\leq N-2}\int_{\R^3}\pa^\al \na_x\times B\cdot \pa^\al E dx
\end{equation}
for some constant $\kappa_2>0$. Then, by taking $\kappa_2>0$ and further $\eps_3>0$ both small enough, it follows from \eqref{thm.efli.p3} and \eqref{thm.efli.p4} that
\begin{multline}\label{thm.efli.p6}
\frac{d}{dt}\CE^{{\rm lin},2}_N(U(t))+\la \sum_{1\leq |\al|\leq N-1}\|\pa^\al E\|^2+
\la \sum_{2\leq |\al|\leq N-1}\|\pa^\al B\|^2\\
\leq C\|\na_xb\|^2+C \sum_{1\leq |\al|\leq N-1}\left(\|\na_x\pa^\al [a,b,c]\|^2+\|\na_x \pa^\al \{\FI-\FP\}u\|^2\right),
\end{multline}
which is the desired  dissipation estimate of the electromagnetic field $E$ and $B$. Now, define
\begin{equation}\label{thm.efli.p7}
   \CE^{{\rm lin}}_N(U(t))=\sum_{|\al|\leq N}(\|\pa^\al u\|^2+\|\pa^\al [E,B]\|^2)
   +\kappa_4\left(\CE^{{\rm lin},1}_N(U(t)) + \kappa_3\CE^{{\rm lin},2}_N(U(t)) \right)
\end{equation}
with constants $\kappa_3>0$ and $\kappa_4>0$, where $\CE^{{\rm lin},1}_N(U(t))$, $\CE^{{\rm lin},2}_N(U(t))$ are defined in \eqref{thm.efli.p1-def} and \eqref{thm.efli.p5}.  In the same way as before, by taking properly small constants $\kappa_3>0$, $\eps_2>0$ and $\kappa_4>0$ in turn, \eqref{thm.efli.1} follows from the linear combination of \eqref{thm.efli.p1}, \eqref{thm.efli.p2} and \eqref{thm.efli.p6}.
Moreover,  it is easy to verify that $\CE^{{\rm lin}}_N(U(t))$ is the desired $L^2$ energy functional satisfying \eqref{thm.efli.00} and $\CD^{{\rm lin}}_N(U(t))$ is given by \eqref{thm.efli.01}. Here, one has to check
\begin{equation}\label{thm.efli.p7.1}
    \sum_{2\leq |\al|\leq N-1}\|\pa^\al B\|^2\leq C\sum_{1\leq |\al|\leq N-2}\|\pa^\al \na_x\times B\|^2.
\end{equation}
In fact, by taking $\al$ with $2\leq |\al|\leq N-1$ and using $B=-\De_x^{-1}\na_x\times\na_x\times B$ due to $\na_x\cdot B=0$, it holds
\begin{equation*}
    \pa^\al B=-\pa^\al \De_x^{-1}\na_x\times\na_x\times B=-\pa^{\al-\ga_i} \pa_i\De_x^{-1}\na_x\times\na_x\times B,
\end{equation*}
for some $\ga_i$ $(1\leq i\leq 3)$ with $|\ga_i|=1$. Since $\pa_i\De_x^{-1}\pa_j$ for any $1\leq i,j\leq 3$ is a bounded operator from $L^p$ to itself with $1<p<\infty$,
\begin{equation*}
   \|\pa^\al B\|\leq C\|\pa^{\al-\ga_i} \na_x\times B\|.
\end{equation*}
Hence, \eqref{thm.efli.p7.1} follows from taking summation of the above inequality over $2\leq |\al|\leq N-1$.

Now, the rest is to prove \eqref{thm.efli.p3} and \eqref{thm.efli.p4}. Take $\al$ with $1\leq |\al|\leq N-1$. By using \eqref{ME.lieq.macro}$_2$ to replace $E$ and then using \eqref{ME.lieq.maxwell}$_1$ to replace $\pa_t E$, one can compute
\begin{eqnarray}\label{thm.efli.p8}
% \nonumber to remove numbering (before each equation)
  \|\pa^\al E\|^2 &=& \int_{\R^3} \pa^\al E\cdot \pa^\al E dx\\
  &=&\int_{\R^3} \pa^\al E\cdot \pa^\al [\pa_t b+\na_x (a+2c)+\na_x\Theta(\{\FI-\FP\}u)] dx\nonumber \\
  &=&\frac{d}{dt}\int_{\R^3}\pa^\al E\cdot \pa^\al b dx -\int_{\R^3}\pa^\al \pa_t E\cdot \pa^\al b dx\nonumber \\
  &&+\int_{\R^3} \pa^\al E\cdot \pa^\al [\na_x (a+2c)+\na_x\Theta(\{\FI-\FP\}u)] dx\nonumber \\
  &=&\frac{d}{dt}\int_{\R^3}\pa^\al E\cdot \pa^\al b dx +\int_{\R^3}\pa^\al (b-\na_x\times B)\cdot \pa^\al b dx\nonumber \\
  &&+\int_{\R^3} \pa^\al E\cdot \pa^\al [\na_x (a+2c)+\na_x\Theta(\{\FI-\FP\}u)] dx.\nonumber
\end{eqnarray}
Then,  \eqref{thm.efli.p3} follows from the above identity  after taking summation over $1\leq |\al|\leq N-1$ and further applying the Cauchy-Schwarz inequality and integration by parts. In fact, it suffices to consider the second term on the r.h.s.~of \eqref{thm.efli.p8}. It can be estimated by
\begin{multline*}
 % \nonumber to remove numbering (before each equation)
   \sum_{1\leq |\al|\leq N-1}\int_{\R^3}\pa^\al (b-\na_x\times B)\cdot \pa^\al b dx =\sum_{1\leq |\al|\leq N-1}\|\pa^\al b\|^2\\
 - \sum_{1\leq |\al|\leq N-2}\int_{\R^3}\pa^\al \na_x\times B\cdot \pa^\al b dx
 + \sum_{|\al|= N-1}\int_{\R^3}\pa^{\al-\ga_i} \na_x\times B\cdot \pa^{\al+\ga_i}  b dx\\
 \leq  \eps_3 \sum_{1\leq |\al|\leq N-2}\|\pa^\al \na_x\times B\|^2+\frac{C}{\eps_3}\sum_{|\al|\leq N-1}\|\na_x\pa^\al b\|^2,
\end{multline*}
where as before $0<\eps_3\leq 1$ is small to be chosen, and $\ga_i$ denotes a multi-index with $|\ga_i|=1$ for some $1\leq i\leq 3$. To prove \eqref{thm.efli.p4}, take $\al$ with $1\leq |\al|\leq N-2$. By using \eqref{ME.lieq.maxwell}$_1$ to replace $\na_x\times B$ and then using  \eqref{ME.lieq.maxwell}$_2$ to replace $\pa_t B$, one has
\begin{multline*}
% \nonumber to remove numbering (before each equation)
  \|\pa^\al \na_x\times B\|^2 = \int_{\R^3}\pa^\al \na_x\times B\cdot \pa^\al \na_x\times B dx
    =\int_{\R^3}\pa^\al \na_x\times B\cdot \pa^\al (\pa_t E+b) dx\\
    =\frac{d}{dt}\int_{\R^3}\pa^\al \na_x\times B\cdot \pa^\al E dx-\int_{\R^3}\pa^\al \na_x\times \pa_t B\cdot \pa^\al E dx+\int_{\R^3}\pa^\al \na_x\times B\cdot \pa^\al b dx\\
    =\frac{d}{dt}\int_{\R^3}\pa^\al \na_x\times B\cdot \pa^\al E dx+\int_{\R^3}\pa^\al \na_x\times \na_x\times E\cdot \pa^\al E dx+\int_{\R^3}\pa^\al \na_x\times B\cdot \pa^\al b dx.
\end{multline*}
Then, \eqref{thm.efli.p4} follows by applying integration by part to the right-hand second term of the above identity, using the Cauchy-Schwarz inequality and then taking summation over $1\leq |\al|\leq N-2$. The proof of Theorem \ref{thm.efli} is complete. \qed

\subsection{$L^2$ time-frequency functional and its optimal dissipation rate}
In this subsection we shall prove Theorem \ref{thm.tfli} as well as Corollary \ref{cor1.tfli} and Corollary \ref{cor2.tfli}. For that, we need to consider the solution $U=[u,E,B]$ to the Cauchy problem \eqref{li.eq}-\eqref{li.ID} in the Fourier space $\R^3_k$. By taking the Fourier transform in $x$ from \eqref{li.eq}$_1$, \eqref{ME.lieq.macro} and
\eqref{ME.lieq.maxwell}, one has
\begin{equation}\label{li.eq.F}
 \pa_t \hat{u}+i\xi\cdot k \hat{u}-\xi\FM^{1/2}\cdot \hat{E}
      =\FL \hat{u} +\hat{h},
\end{equation}
\begin{equation}\label{ME.lieq.macro.F}
\left\{\begin{split}
&\dis \pa_t \hat{a}+ik\cdot \hat{b}=0,\\
&\dis \pa_t \hat{b} +ik (\hat{a}+2\hat{c})+ik \cdot\Theta(\{\FI-\FP\}\hat{u})
    -\hat{E} =0,\\
&\dis \pa_t \hat{c}+\frac{1}{3}ik\cdot \hat{b}+\frac{5}{3}ik \cdot  \Lambda(\{\FI-\FP\}\hat{u})=0,\\
&\dis \pa_t \Theta_{ij}(\{\FI-\FP\}\hat{u})+ik_i \hat{b}_j+ik_j \hat{b}_i-\frac{2}{3}\de_{ij}ik\cdot \hat{b}-\frac{10}{3}\de_{ij}ik\cdot \Lambda(\{\FI-\FP\}\hat{u})=\Theta_{ij}(\hat{\ell}+\hat{h}),\\
&\pa_t \Lambda_i(\{\FI-\FP\}\hat{u})+ik_i  \hat{c}=\Lambda_i(\hat{\ell}+\hat{h}),
  \end{split}\right.
\end{equation}
and
\begin{equation}\label{li.Maxwell.F}
    \left\{\begin{split}
      &\dis\pa_t \hat{E}-ik\times \hat{B}=-\hat{b},\\
      &\dis\pa_t \hat{B} +ik \times \hat{E} =0,\\
     & \dis ik\cdot \hat{E}=\hat{a},\ \ ik\cdot \hat{B}=0,
    \end{split} \right.
\end{equation}
where $\hat{\ell}$ is given by
\begin{equation*}
     \hat{\ell} = -ik\cdot \xi\{\FI-\FP\}\hat{u}+\FL \hat{u}.
\end{equation*}
These equations above are ones to be used through this subsection.

\medskip

\noindent{\bf Proof of Theorem \ref{thm.tfli}:} It is similar to the proof of Theorem \ref{thm.efli}. The difference is that all calculations  are made in the Fourier space. Thus, some details in the following proof will be omitted for simplicity. First of all, as in \cite{D-Hypo}, on one hand, from \eqref{li.eq.F} and \eqref{li.Maxwell.F}, one has
\begin{equation}\label{thm.tfli.p1}
    \frac{1}{2}\pa_t(\|\hat{u}\|_{L^2_\xi}^2+|[\hat{E},\hat{B}]|^2)
    +\la \|\nu^{1/2}\{\FI-\FP\}\hat{u}\|_{L^2_\xi}^2\leq
    C\|\nu^{-1/2}\hat{g}\|_{L^2_\xi}^2,
\end{equation}
and on the other hand, from \eqref{ME.lieq.macro.F}, the following three estimates hold true:
\begin{multline}\label{thm.tfli.p2}
\pa_t\MRk(ik\hat{c}\mid \Lambda(\{\FI-\FP\}\hat{u})
+\la |k|^2|\hat{c}|^2\\
\leq \eps_1 \|k\cdot \hat{b}\|^2
+\frac{C}{\eps_1}(1+|k|^2)
\|\{\FI-\FP\}\hat{u}\|_{L^2_\xi}^2+C\|\nu^{-1/2}\hat{g}\|_{L^2_\xi}^2,
\end{multline}
\begin{multline}\label{thm.tfli.p3}
\pa_t\MRk\sum_{ij=1}^3(ik_i\hat{b}_j+ik_j\hat{b}_i-\frac{2}{3}\de_{ij}ik\cdot \hat{b}\mid
\Theta_{ij}(\{\FI-\FP\}\hat{u})
+\la |k|^2|\hat{b}|^2 \\
\leq \eps_2|k|^2|[\hat{a},\hat{c}]|^2+\eps_2\frac{|k|^2}{1+|k|^2}|\hat{E}|^2
+\frac{C}{\eps_2}(1+|k|^2)
\|\{\FI-\FP\}\hat{u}\|_{L^2_\xi}^2+C\|\nu^{-1/2}\hat{g}\|_{L^2_\xi}^2,
\end{multline}
and
\begin{equation}\label{thm.tfli.p4}
\pa_t\MRk(ik\hat{a}\mid \hat{b}) +\la (1+|k|^2)|\hat{a}|^2
\leq |k\cdot \hat{b}|^2+C|k|^2|\hat{c}|^2+C|k|^2\|\{\FI-\FP\}\hat{u}\|_{L^2_\xi}^2,
\end{equation}
where constants $0<\eps_1,\eps_2\leq 1$ are to be chosen. Define
\begin{multline}\label{thm.tfli.p5}
% \nonumber to remove numbering (before each equation)
  \CE^{{\rm lin},1}(\hat{U}(t)) =\frac{1}{1+|k|^2}\MRk\{
(ik\hat{c}\mid \Lambda(\{\FI-\FP\}\hat{u})\\
  +\sum_{ij=1}^3(ik_i\hat{b}_j+ik_j\hat{b}_i-\frac{2}{3}\de_{ij}ik\cdot \hat{b}\mid
\Theta_{ij}(\{\FI-\FP\}\hat{u})
+\kappa_1(ik\hat{a}\mid \hat{b})\}
\end{multline}
for a constant $\kappa_1>0$. One  can take $\kappa_1$ and then $\eps_1$ both small enough such that the sum of \eqref{thm.tfli.p2}, \eqref{thm.tfli.p3} and $\kappa_1\times$\eqref{thm.tfli.p4} gives
\begin{multline}\label{thm.tfli.p6}
\pa_t \CE^{{\rm lin},1}(\hat{U}(t))+\la \frac{|k|^2}{1+|k|^2}|[\hat{a},\hat{b},\hat{c}]|^2+|\hat{a}|^2\\
\leq \eps_2\frac{|k|^2}{(1+|k|^2)^2}|\hat{E}|^2+\frac{C}{\eps_2} \left(
\|\{\FI-\FP\}\hat{u}\|_{L^2_\xi}^2+\|\nu^{-1/2}\hat{g}\|_{L^2_\xi}^2\right).
\end{multline}

For estimates on the dissipation of $\hat{E}$ and $\hat{B}$, it is straightforward to deduce from \eqref{li.eq.F}$_2$ and \eqref{ME.lieq.macro.F} the following two identities
\begin{multline}\label{thm.tfli.p7}
-\pa_t (k\times \hat{E}\mid k\times \hat{b})+|k\times \hat{E}|^2=|k\times \hat{b}|^2-(ik\times k\times \hat{B}\mid k\times \hat{b})\\
+(k\times \hat{E}\mid ik\times (k\cdot \Theta(\{\FI-\FP\}\hat{u}))),
\end{multline}
and
\begin{equation}\label{thm.tfli.p8}
-\pa_t (ik\times \hat{B}\mid \hat{E})+|k\times \hat{B}|^2=|k\times \hat{E}|^2+(ik\times \hat{B}\mid \hat{b}).
\end{equation}
By applying the Cauchy-Schwarz to \eqref{thm.tfli.p8} and then multiplying it by $|k|^2/(1+|k|^2)^3$, one has
\begin{equation*}
-\pa_t \frac{|k|^2}{(1+|k|^2)^3}\MRk (ik\times \hat{B}\mid \hat{E})+\la \frac{|k|^2|k\times \hat{B}|^2}{(1+|k|^2)^3}\leq \frac{|k|^2|k\times \hat{E}|^2}{(1+|k|^2)^3}+C\frac{|k|^2| \hat{b}|^2}{(1+|k|^2)^3},
\end{equation*}
which implies
\begin{equation}\label{thm.tfli.p9}
-\pa_t \frac{|k|^2}{(1+|k|^2)^3}\MRk (ik\times \hat{B}\mid \hat{E})+\la \frac{|k|^2|k\times \hat{B}|^2}{(1+|k|^2)^3}\leq \frac{|k\times \hat{E}|^2}{(1+|k|^2)^2}+C\frac{|k|^2| \hat{b}|^2}{1+|k|^2}.
\end{equation}
Similarly, after dividing \eqref{thm.tfli.p7} by $(1+|k|^2)^2$ and then using  Cauchy-Schwarz,
\begin{eqnarray}\label{thm.tfli.p10}
&&-\frac{\pa_t \MRk (k\times \hat{E}\mid k\times \hat{b})}{(1+|k|^2)^2}+\la \frac{|k\times \hat{E}|^2}{(1+|k|^2)^2}\\
&&\leq \frac{|k\times \hat{b}|^2}{(1+|k|^2)^2}
+\frac{|k\times k\times \hat{B}|\cdot |k\times \hat{b}|}{(1+|k|^2)^2}
+C\frac{|k|^4 |\Theta(\{\FI-\FP\}\hat{u})|^2}{(1+|k|^2)^2}\nonumber\\
&&\leq \eps_3 \frac{|k|^2|k\times \hat{B}|^2}{(1+|k|^2)^3}+\frac{C}{\eps_3}\frac{|k|^2}{1+|k|^2}|\hat{b}|^2
+C\|\{\FI-\FP\}\hat{u}\|_{L^2_\xi}^2,\nonumber
\end{eqnarray}
where we used the inequality
\begin{equation*}
    \frac{|k\times k\times \hat{B}|\cdot |k\times \hat{b}|}{(1+|k|^2)^2}
    \leq \eps_3  \frac{|k\times k\times \hat{B}|^2}{(1+|k|^2)^3}+\frac{C}{\eps_3}\frac{|k\times \hat{b}|^2}{1+|k|^2}
\end{equation*}
for an arbitrary constant $0<\eps_3\leq 1$. Then, in terms of \eqref{thm.tfli.p9} and \eqref{thm.tfli.p10}, let us define
\begin{equation}\label{thm.tfli.p11}
% \nonumber to remove numbering (before each equation)
  \CE^{{\rm lin},2}(\hat{U}(t)) =-\frac{\MRk (k\times \hat{E}\mid k\times \hat{b})}{(1+|k|^2)^2}
  -\kappa_2\frac{|k|^2\MRk (ik\times \hat{B}\mid \hat{E})}{(1+|k|^2)^3},
\end{equation}
where $\kappa_2>0$ is chosen small enough such that
\begin{equation}\label{thm.tfli.p12}
    \pa_t \CE^{{\rm lin},2}(\hat{U}(t))+\la \frac{|k\times \hat{E}|^2}{(1+|k|^2)^2}
    +\la \frac{|k|^4|\hat{B}|^2}{(1+|k|^2)^3}\leq C\frac{|k|^2| \hat{b}|^2}{1+|k|^2}+C\|\{\FI-\FP\}\hat{u}\|_{L^2_\xi}^2.
\end{equation}
Here, we used $|k\times \hat{B}|=|k|\cdot |k\times \hat{B}|$ due to $k\cdot \hat{B}=0$.

Now, in terms of \eqref{thm.tfli.p1}, \eqref{thm.tfli.p6} with $0<\eps_2\leq 1$ and \eqref{thm.tfli.p12},  we define
\begin{equation}\label{def.lin.tk}
  \CE^{{\rm lin}}(\hat{U}(t))=\|\hat{u}\|_{L^2_\xi}^2+|[\hat{E},\hat{B}]|^2
   +\kappa_4\left( \CE^{{\rm lin},1}(\hat{U}(t)) +  \kappa_3\CE^{{\rm lin},2}(\hat{U}(t)) \right),
\end{equation}
where $\CE^{{\rm lin},1}(\hat{U}(t))$, $\CE^{{\rm lin},2}(\hat{U}(t))$ are denoted by \eqref{thm.tfli.p5} and \eqref{thm.tfli.p11}, and  $\kappa_3>0$, $\eps_2>0$ and $\kappa_4>0$ are chosen in turn small enough such that \eqref{thm.tfli.1} and \eqref{thm.tfli.3} hold true and  $ \CD^{{\rm lin}}(\hat{U}(t))$ is given by \eqref{thm.tfli.2}. Here, notice that we used
\begin{equation*}
|k|^2|\hat{E}|^2=|k\cdot \hat{E}|^2+|k\times \hat{E}|^2=|\hat{a}|^2+|k\times \hat{E}|^2.
\end{equation*}
The proof of Theorem  \ref{thm.tfli} is complete. \qed

\medskip

\noindent{\bf Proof of Corollary \ref{cor1.tfli} and Corollary \ref{cor2.tfli}:} First, \eqref{cor1.tfli.1} in Corollary \ref{cor1.tfli} immediately results from \eqref{thm.tfli.3} by noticing
\begin{equation*}
    \CD(\hat{U}(t,k))\geq \la \frac{|k|^4}{(1+|k|^2)^3}\CE(\hat{U}(t,k))
\end{equation*}
due to the definitions \eqref{thm.tfli.1}, \eqref{thm.tfli.2} of $\CE(\hat{U}(t,k))$, $\CD(\hat{U}(t,k))$. For Corollary \ref{cor2.tfli}, it suffices to define
\begin{equation}\label{def.lin.morder}
    \CE_m^{{\rm lin}}(\hat{U}(t,k))=\sum_{|\al|=m}\CE(\widehat{\na_x^kU}(t,k))=\sum_{|\al|=m}\CE((ik)^\al\hat{U}(t,k)).
\end{equation}
Since the system \eqref{li.eq} satisfied by $U=[u,E,B]$ is linear,
it is easy to see from Theorem \ref{thm.tfli} that $\CE_m^{{\rm
lin}}(\hat{U}(t,k))$ satisfies \eqref{cor2.tfli.1} and
\eqref{cor2.tfli.3} with $\CD_m^{{\rm lin}}(\hat{U}(t,k))$ given by
\eqref{cor2.tfli.2}. This hence completes the proof of Corollary
\ref{cor2.tfli}. \qed

\subsection{Time-decay estimates}
In this subsection we turn to the proof of Theorem \ref{thm.ls}.
Theorem \ref{thm.ls} actually follows from \eqref{cor1.tfli.1} in
Corollary \ref{cor1.tfli}. However, we would rather provide a
similar much more general result. Recall the definition of
$\semiG(t)$ as in \eqref{def.lu1}. In what follows, for the Fourier
transform $\hat{U}(t,k)$ of $U=[u,E,B]$, we set
\begin{equation}\label{def.abs}
    |\hat{U}(t,k)|=\|\hat{u}\|_{L^2_\xi}+|[\hat{E},\hat{B}]|
\end{equation}
for simplicity.

\begin{lemma}\label{lem.glidecay}
Assume that for any initial data $U_0$, the linear homogeneous
solution $U(t)=\semiG(t)U_0$ obeys the pointwise estimate
\begin{equation}\label{lem.glidecay.1}
    |\hat{U}(t,k)|\leq C e^{-\phi(k) t}|\hat{U}_0(k)|
\end{equation}
for all $t\geq 0$, $k\in \R^3$, where $\phi(k)$ is a strictly positive, continuous and real-valued function over $k\in \R^3$ and satisfies
\begin{equation}\label{lem.glidecay.2}
    \phi(k)\to \left\{
    \begin{array}{ll}
      O(1)|k|^{\si_+} & \ \ \text{as $|k|\to 0$},\\[3mm]
      O(1)|k|^{-\si_-} & \ \ \text{as $|k|\to \infty$},
    \end{array}\right.
\end{equation}
for two constants $\si_->\si_+>0$. Let $m\geq 0$ be an integer, $1\leq p,r\leq 2\leq q\leq \infty$ and $\si\geq 0$. Then, $U(t)=\semiG(t)U_0$ obeys the time-decay estimate
\begin{equation}\label{lem.glidecay.3}
 \|\na_x^m U(t)\|_{\CZ_q}
\leq
C(1+t)^{-\frac{3}{\si_+}(\frac{1}{p}-\frac{1}{q})-\frac{m}{\si_+}}
\|U_0\|_{\CZ_p}
 +C(1+t)^{-\frac{\si}{\si_-}} \|\na_x^{m+[\si+3(\frac{1}{r}-\frac{1}{q})]_+}U_0\|_{\CZ_r},
\end{equation}
for any $t\geq 0$,  {where $[\cdot]_+$ is defined in \eqref{def.plus}.}
\end{lemma}

\begin{proof}
Take a constant $R>0$. From the assumptions on $\phi(k)$, it is easy to see
\begin{equation*}
    \phi(k)\geq \left\{
    \begin{array}{ll}
      \la |k|^{\si_+} & \ \ \text{if $|k|\leq R$},\\[3mm]
      \la |k|^{-\si_-} & \ \ \text{if $|k|\geq R$}.
    \end{array}\right.
\end{equation*}
Take $2\leq q\leq \infty$ and an integer $m\geq 0$. From Hausdorff-Young inequality,
\begin{multline}\label{lem.glidecay.p1}
\|\na_x^m U(t)\|_{\CZ_q}\leq C\left\||k|^m e^{-\phi(k)t}\hat{U}_0\right\|_{\CZ_{q'}}\\
\leq C\left\||k|^m e^{-\la |k|^{\si_+}t}\hat{U}_0\right\|_{\CZ_{q'}(|k|\leq R)}
+C \left\||k|^m e^{-\la |k|^{-\si_-}t}\hat{U}_0\right\|_{\CZ_{q'}(|k|\geq R)}:=I_1+I_2,
\end{multline}
where $\frac{1}{q}+\frac{1}{q'}=1$. For $I_1$, by the definition of
the norm $\|\cdot\|_{\CZ_{q'}}$,
\begin{equation*}
    I_1=C\left\||k|^m e^{-\la |k|^{\si_+}t}\hat{u}_0\right\|_{L^2_\xi(L^{q'}(|k|\leq R))}
    +C\left\||k|^m e^{-\la |k|^{\si_+}t}[\hat{E}_0,\hat{B}_0]\right\|_{L^{q'}(|k|\leq
    R)}.
\end{equation*}
Here, note that since $1\leq q'\leq 2$, from the Minkowski
inequality,
\begin{multline*}
\left\||k|^m e^{-\la
|k|^{\si_+}t}\hat{u}_0\right\|_{L^2_\xi(L^{q'}(|k|\leq R))} \leq
\left\||k|^m e^{-\la |k|^{\si_+}t}\hat{u}_0\right\|_{L^{q'}(|k|\leq
R;L^2_\xi)}\\
=\left\||k|^m e^{-\la
|k|^{\si_+}t}\|\hat{u}_0\|_{L^2_\xi}\right\|_{L^{q'}(|k|\leq R)}.
\end{multline*}
Hence, we arrive at
\begin{equation*}
    I_1\leq C\left\||k|^m e^{-\la |k|^{\si_+}t}[\|\hat{u}_0\|_{L^2_\xi},
    \hat{E}_0,\hat{B}_0]\right\|_{L^{q'}(|k|\leq
    R)}.
\end{equation*}
Take $1\leq p\leq 2$. Further using the H\"{o}lder inequality for
$\frac{1}{q'}=\frac{p'-q'}{p'q'}+\frac{1}{p'}$ with $p'$ given by
$\frac{1}{p}+\frac{1}{p'}=1$,
\begin{equation*}
    I_1\leq C\left\||k|^m e^{-\la |k|^{\si_+}t}\right\|_{L^{\frac{p'q'}{p'-q'}}(|k|\leq R)}
    \left\|[\|\hat{u}_0\|_{L^2_\xi},
    \hat{E}_0,\hat{B}_0]\right\|_{L^{p'}(|k|\leq
    R)}
\end{equation*}
Here, the right-hand first term can be estimated in a standard way \cite{Ka}
as
\begin{equation*}
\left\||k|^m e^{-\la
|k|^{\si_+}t}\right\|_{L^{\frac{p'q'}{p'-q'}}(|k|\leq R)}\leq C
(1+t)^{-\frac{3}{\si_+}(\frac{1}{p}-\frac{1}{q})-\frac{m}{\si_+}}
\end{equation*}
by using change of variable $kt^{\frac{1}{\si_+}}\to k$, and the
right-hand second term is estimated by Minkowski and Hausdorff-Young
inequalities as
\begin{multline*}
\left\|[\|\hat{u}_0\|_{L^2_\xi},
    \hat{E}_0,\hat{B}_0]\right\|_{L^{p'}(|k|\leq
    R)}=\|\hat{u}_0\|_{L^{p'}(|k|\leq R;L^2_\xi)}+\left\|[
    \hat{E}_0,\hat{B}_0]\right\|_{L^{p'}(|k|\leq
    R)}\\
    \leq\|\hat{u}_0\|_{L^2_\xi(L^{p'}(|k|\leq R))}+\left\|[
    \hat{E}_0,\hat{B}_0]\right\|_{L^{p'}(|k|\leq
    R)}\\
    \leq
    C(\|u_0\|_{L^2_\xi(L^p_x)}+\|[E_0,B_0]\|_{L^p})=C\|U_0\|_{\CZ_p},
\end{multline*}
where the Minkowski inequality was validly used due to $p'\geq 2$.
Therefore, for $I_1$, one has
\begin{equation*}
I_1\leq
C(1+t)^{-\frac{3}{\si_+}(\frac{1}{p}-\frac{1}{q})-\frac{m}{\si_+}}\|U_0\|_{\CZ_p}.
\end{equation*}
To estimate $I_2$, take a constant $\si\geq 0$ so that
\begin{equation*}
    I_2=C \left\||k|^m e^{-\la |k|^{-\si_-}t}\hat{U}_0\right\|_{\CZ_{q'}(|k|\geq R)}
    \leq C\sup_{|k|\geq R}|k|^{-\si} e^{-\la |k|^{-\si_-}t}
    \left\||k|^{m+\si} \hat{U}_0\right\|_{\CZ_{q'}(|k|\geq R)}.
\end{equation*}
Here, the right-hand first term decays in time as
\begin{equation*}
\sup_{|k|\geq R}|k|^{-\si} e^{-\la |k|^{-\si_-}t}\leq
C(1+t)^{-\frac{\si}{\si_-}}.
\end{equation*}
We estimate the right-hand second term as follows. Take $1\leq
r\leq 2$ with $\frac{1}{r}+\frac{1}{r'}=1$ and take a constant
$\eps>0$ small enough. Then, similarly as before, from Minkowski and
H\"{o}lder inequalities for
$\frac{1}{q'}=\frac{r'-q'}{r'q'}+\frac{1}{r'}$, one has
\begin{multline*}
 \left\||k|^{m+\si} \hat{U}_0\right\|_{\CZ_{q'}(|k|\geq R)}\leq
  \left\||k|^{m+\si}[\|\hat{u}_0\|_{L^2_\xi},\hat{E}_0,\hat{B}_0]\right\|_{L^{q'}(|k|\geq
  R)}\\
  \leq \left\||k|^{-3(1+\eps)\frac{r'-q'}{r'q'}}\right\|_{L^{\frac{r'q'}{r'-q'}}(|k|\geq
  R)}\left\||k|^{m+\si+3(1+\eps)\frac{r'-q'}{r'q'}}
  [\|\hat{u}_0\|_{L^2_\xi},\hat{E}_0,\hat{B}_0]\right\|_{{L^{r'}(|k|\geq
  R)}}\\
  \leq C_\eps\left\||k|^{m+[\si+3(\frac{1}{r}-\frac{1}{q})]_+}
  [\|\hat{u}_0\|_{L^2_\xi},\hat{E}_0,\hat{B}_0]\right\|_{{L^{r'}(|k|\geq
  R)}}.
\end{multline*}
Here, by  Minkowski inequality due to $q'\geq 2$ once again and
further by Hausdorff-Young inequality,
\begin{multline*}
\left\||k|^{m+[\si+3(\frac{1}{r}-\frac{1}{q})]_+}
  [\|\hat{u}_0\|_{L^2_\xi},\hat{E}_0,\hat{B}_0]\right\|_{{L^{r'}(|k|\geq
  R)}}\\
\leq \left\||k|^{m+[\si+3(\frac{1}{r}-\frac{1}{q})]_+}
\hat{u}_0\right\|_{L^2_\xi(L^{r'}(|k|\geq
R))}+\left\||k|^{m+[\si+3(\frac{1}{r}-\frac{1}{q})]_+}
  [\hat{E}_0,\hat{B}_0]\right\|_{{L^{r'}(|k|\geq
  R)}}\\
  \leq
  C\|\na_x^{m+[\si+3(\frac{1}{r}-\frac{1}{q})]_+}U_0\|_{\CZ_r}.
\end{multline*}
Then, it follows that
\begin{equation*}
\left\||k|^{m+\si} \hat{U}_0\right\|_{\CZ_{q'}(|k|\geq R)}\leq
C\|\na_x^{m+[\si+3(\frac{1}{r}-\frac{1}{q})]_+}U_0\|_{\CZ_r}.
\end{equation*}
Thus, $I_2$ is estimated by
\begin{equation*}
    I_2\leq
    C(1+t)^{-\frac{\si}{\si_-}}\|\na_x^{m+[\si+3(\frac{1}{r}-\frac{1}{q})]_+}U_0\|_{\CZ_r}.
\end{equation*}
Now, \eqref{lem.glidecay.3} follows by plugging the estimates of
$I_1$ and $I_2$ into \eqref{lem.glidecay.p1}. This completes the
proof of Lemma \ref{lem.glidecay}.
\end{proof}

\noindent{\bf Proof of Theorem \ref{thm.ls}:} To prove
\eqref{thm.ls.1}, by  letting $h=0$ and using Corollary
\ref{cor1.tfli},
\begin{equation*}
    \CE(\hat{U}(t,k))\leq e^{-\frac{\la
    |k|^4}{(1+|k|^2)^3}t}\CE(\hat{U}_0(k))
\end{equation*}
for any $t\geq 0$ and $k\in \R^3$, where we have set $U=U^{I}$ for
simplicity. Due to \eqref{thm.tfli.1} and \eqref{def.abs},
$\CE(\hat{U}(t,k))\sim |\hat{U}(t,k)|^2$ holds so that
\begin{equation*}
    |\hat{U}(t,k)|\leq C e^{-\frac{\la
    |k|^4}{2(1+|k|^2)^3}t}|\hat{U}_0(k)|
\end{equation*}
for any $t\geq 0$ and $k\in \R^3$. This shows that corresponding to
\eqref{lem.glidecay.1} and \eqref{lem.glidecay.2} of Lemma
\ref{lem.glidecay}, one has the special situation
\begin{equation*}
    \phi(k)=\frac{\la
    |k|^4}{2(1+|k|^2)^3}
\end{equation*}
with $\si_+=4$, $\si_-=2$. Thus, one can apply Lemma
\ref{lem.glidecay} to obtain \eqref{thm.ls.1} from
\eqref{lem.glidecay.3}.

To prove \eqref{thm.ls.2}, we let $U_0=0$, and also set $U=U^{II}$
for simplicity. Note that \eqref{cor1.tfli.1} implies
\begin{equation*}
    |\hat{U}(t,k)|^2\leq C\int_0^t e^{-\frac{\la
    |k|^4}{2(1+|k|^2)^3}(t-s)}\|\nu^{-1/2}\hat{h}(s)\|_{L^2_\xi}^2ds
\end{equation*}
for any $t\geq 0$ and $k\in \R^3$. One can again apply Lemma
\ref{lem.glidecay} with $q=r=2$ so that \eqref{thm.ls.2} follows.
The proof of Theorem  \ref{thm.ls} is complete. \qed

\section{Nonlinear system}\label{sec.nonlinear}

In this section we are concerned with the global existence of
solutions to the Cauchy problem \eqref{VMB.eq}-\eqref{VMB.ID} of the
reformulated nonlinear Vlasov-Maxwell-Boltzmann system. We first devote ourselves to the
proof of some uniform-in-time a priori estimates on the solution. In
what follows, $U=[u,E,B]$ is supposed to be smooth in all arguments
and satisfy the  system   \eqref{VMB.eq} over $0\leq t\leq T$ for
some $0<T\leq \infty$.

\begin{lemma}\label{lem.non}
Under the assumption that
$
\sup_{0\leq t\leq T}\|(a+2c)(t)\|_{L^\infty}
$
is small enough, there is $\CE_N(U(t))$ satisfying \eqref{thm.non.1}
such that
\begin{equation}\label{lem.non.1}
    \frac{d}{dt}\CE_N(U(t))+\la \CD_N(U(t))\leq C (\CE_N(U(t))^{1/2}+\CE_N(U(t)))\CD_N(U(t))
\end{equation}
for any $0\leq t\leq T$, where $\CD_N(U(t))$ is defined in
\eqref{thm.non.2}.
\end{lemma}

\begin{proof}
First, the zero-order energy estimate implies
\begin{multline}\label{lem.non.p01}
\frac{1}{2}\frac{d}{dt}(\|u\|^2+\|[E,B]\|^2-\int
|b|^2(a+2c)dx+\la\|\nu^{1/2}\{\FI-\FP\} u\|^2\\
 \leq
C(\CE_N(U(t))^{1/2}+\CE_N(U(t)))\CD_N(U(t)).
\end{multline}
Here and hereafter, when $\CE_N(U(t))$ occurs in the right-hand
terms of inequalities, it means an equivalent energy functional
satisfying \eqref{thm.non.1} and its explicit representation will be
determined later on. In fact, from the system \eqref{VMB.eq},
\begin{equation}\label{lem.non.p02}
\frac{1}{2}\frac{d}{dt}(\|u\|^2+\|[E,B]\|^2)+\la
\|\nu^{1/2}\{\FI-\FP\} u\|^2\leq \iint
u\Ga(u,u)dxd\xi+\frac{1}{2}\iint \xi\cdot E u^2 dxd\xi.
\end{equation}
Here, for the first term on the r.h.s.~of \eqref{lem.non.p02}, it is
a standard fact as in \cite{Guo-IUMJ} or \cite{Guo2} that it is bounded by
$C\CE_N(U(t))^{1/2}\CD_N(U(t))$. Since all estimates on terms
involving the nonlinear term $\Ga(u,u)$ in the following can be
handled in the similar way, we shall omit the details of their proof
for brevity. The right-hand second term of \eqref{lem.non.p02} can
be estimated as in \cite[Lemma 4.4]{DY-09VPB}:
\begin{multline*}
\frac{1}{2}\iint \xi\cdot E u^2 dxd\xi=\frac{1}{2}\iint \xi\cdot E
|\FP u|^2 dxd\xi+\iint \xi\cdot E \FP u\{\FI-\FP\}u
dxd\xi\\
+\frac{1}{2}\iint \xi\cdot E |\{\FI-\FP\}u|^2 dxd\xi.
\end{multline*}
Here, it is easy to see
\begin{multline*}
\iint \xi\cdot E \FP u\{\FI-\FP\}u dxd\xi +\frac{1}{2}\iint \xi\cdot
E |\{\FI-\FP\}u|^2 dxd\xi\\
\leq
C\|E\|_{H^2}(\|\na_x[a,b,c]\|^2+\|\nu^{1/2}\{\FI-\FP\}u\|^2)\leq
C\CE_N(U(t))^{1/2}\CD_N(U(t)).
\end{multline*}
And, as in \cite[Lemma 4.4]{DY-09VPB}, one can compute
\begin{equation*}
\frac{1}{2}\iint \xi\cdot E |\FP u|^2 dxd\xi=\int E\cdot b (a+2c)dx,
\end{equation*}
where by replacing $E$ by equation \eqref{ME.eq.macro}$_2$, using
integration by part in $t$ and then replacing $\pa_t (a+2c)$ by
equations \eqref{ME.eq.macro}$_1$ and \eqref{ME.eq.macro}$_3$, gives
\begin{equation*}
\begin{split}
\int E\cdot b (a+2c)dx= &\,\frac{1}{2}\frac{d}{dt}\int
|b|^2(a+2c)dx\\
&+\int |b|^2[\frac{5}{6}\na_x\cdot b+\frac{1}{6}\na_x\cdot
\Lambda(\{\FI-\FP\}u)-\frac{1}{3}E\cdot
b]dx\\
&+\int [\na_x(a+2c)+\na_x\Theta(\{\FI-\FP\}u)]\cdot b (a+2c)dx\\
& -\int Ea\cdot b (a+2c)dx-\int b\times B\cdot b (a+2c)dx.
\end{split}
\end{equation*}
Note $b\times B\cdot b=0$. Hence, it follows that
\begin{multline*}
\frac{1}{2}\iint \xi\cdot E |\FP u|^2
dxd\xi\leq\frac{1}{2}\frac{d}{dt}\int |b|^2(a+2c)dx\\
+C\|[a,b,c,E]\|_{H^1}(\|\na_x [a,b,c]\|^2+\|\na_x\{\FI-\FP\}u\|^2)
+C\|E\|\, \|\na_x a\|\,\|\na_x [a,b,c]\|^2\\
\leq\frac{1}{2}\frac{d}{dt}\int |b|^2(a+2c)dx+
C(\CE_N(U(t))^{1/2}+\CE_N(U(t)))\CD_N(U(t)).
\end{multline*}
Collecting the above estimates and putting them into
\eqref{lem.non.p02} proves \eqref{lem.non.p01}.

Next, for the estimates on all derivatives including the pure
spatial derivatives and space-velocity mixed derivatives, one has
\begin{multline}\label{lem.non.p03}
\frac{1}{2}\frac{d}{dt}\sum_{1\leq |\al|\leq N}(\|\pa^\al
u\|^2+\|\pa^\al [E,B]\|^2)+\la \sum_{1\leq |\al|\leq
N}\|\nu^{1/2}\pa^\al \{\FI-\FP\}u\|^2\\
\leq C\CE_N(U(t))^{1/2}\CD_N(U(t)),
\end{multline}
and
\begin{multline}\label{lem.non.p04}
\frac{1}{2}\frac{d}{dt}\sum_{k=1}^NC_k\sum_{\substack{|\be|=k \\
|\al|+|\be|\leq N}}\|\pa^\al_\be \{\FI-\FP\}u\|^2+\la\sum_{\substack{|\be|\geq 1\\
|\al|+|\be|\leq N}} \|\nu^{1/2}\pa^\al_\be \{\FI-\FP\}u\|^2\\
\leq C\CE_N(U(t))^{1/2}\CD_N(U(t))+C\sum_{|\al|\leq N-1}\|\pa^\al
\na_x [a,b,c]\|^2 +C\sum_{|\al|\leq N}\|\nu^{1/2}\pa^\al
\{\FI-\FP\}u\|^2,
\end{multline}
where $C_k$ $(1\leq k\leq N)$ are strictly positive constants. Since
the nonlinear term $g$ takes the form as in \eqref{def.ell.g}, the
proof of \eqref{lem.non.p03} and \eqref{lem.non.p04} is almost the
same as in \cite{DY-09VPB} for the case of the
Vlasov-Poisson-Boltzmann system and thus details are omitted for
brevity.

Finally, the key step is to obtain the macroscopic dissipation of
$a,b,c$ and $E,B$ for the nonlinear system \eqref{VMB.eq}. This is
similar to the proof of Theorem \ref{thm.efli} in Subsection \ref{sec.ef}
for the linearized system. Here, the additional efforts should be
made to take care of all quadratically nonlinear terms in $g$
defined by \eqref{def.ell.g}. But, these nonlinear estimates once
again are almost the same as in \cite{DY-09VPB} for the case of the
Vlasov-Poisson-Boltzmann system so we omit details for brevity.
Thus, we have the following estimates. Recall the definitions
\eqref{thm.efli.p1-def} and \eqref{thm.efli.p5} of two interactive
functionals $\CE_N^{{\rm lin},1}(U(t))$ and  $\CE_N^{{\rm
lin},2}(U(t))$, where $\kappa_1>0$, $\kappa_2>0$ in
\eqref{thm.efli.p1-def} and \eqref{thm.efli.p5} are sufficiently
small. It turns out that
\begin{multline}\label{lem.non.p05}
\frac{d}{dt}\left(\CE_N^{{\rm lin},1}(U(t)) +\kappa_3\CE_N^{{\rm
lin},2}(U(t))\right)+\la\sum_{|\al|\leq N} \|\pa^\al a\|^2+\la
\sum_{1\leq |\al|\leq N}\|\pa^\al [b,c]\|^2\\
+\la \sum_{1\leq |\al|\leq N-1}\|\pa^\al E\|^2 +\la \sum_{2\leq
|\al|\leq N-1}\|\pa^\al B\|^2\\
\leq C\sum_{|\al|\leq N} \|\pa^\al
\{\FI-\FP\}u \|^2+C\CE_N(U(t))\CD_N(U(t)),
\end{multline}
where $\kappa_3>0$ is small enough. This is the desired estimate on
the macroscopic dissipation.

Now, we are in a position to prove \eqref{lem.non.1}. Let
$\kappa_4>0$ be taken as in \eqref{thm.efli.p7}.  Define
\begin{eqnarray}\label{lem.non.p06}
% \nonumber to remove numbering (before each equation)
  \CE_N(U(t)) &=& \CE_N^{{\rm lin}}(U(t))-\int
|b|^2(a+2c)dx+\kappa_5\sum_{k=1}^NC_k\sum_{\substack{|\be|=k \\
|\al|+|\be|\leq N}}\|\pa^\al_\be \{\FI-\FP\}u\|^2,
\end{eqnarray}
where $\kappa_5>0$ is a constant to be chosen. Here, notice that by
recalling the definition \eqref{thm.efli.p7} of $\CE_N^{{\rm
lin}}(U(t))$, $\CE_N(U(t))$ can be rewritten as
\begin{eqnarray}\label{lem.non.p07}
% \nonumber to remove numbering (before each equation)
 &&\CE_N(U(t)) \\
&& =\sum_{ |\al|\leq N}(\|\pa^\al
u\|^2+\|\pa^\al [E,B]\|^2)-\int |b|^2(a+2c)dx\nonumber\\
&&\ \ \ \ +\kappa_4 \left\{
\sum_{|\al|\leq N-1}\int_{\R^3}\na_x\pa^\al c\cdot \Lambda(\pa^\al\{\FI-\FP\}u)dx\right. \nonumber\\
 & &\ \ \ \ \ \ \ \ \ \ \ \ \  +\sum_{\substack{1\leq i,j\leq 3 \nonumber\\
|\al|\leq N-1}}\int_{\R^3} (\pa_i\pa^\al b_j+\pa_j\pa^\al
b_i-\frac{2}{3}\de_{ij}\na_x\cdot \pa^\al b)
\Theta_{ij}(\pa^\al\{\FI-\FP\}u)dx\nonumber\\
&&\ \ \ \ \ \ \ \ \ \ \ \ \   \left.+\kappa_1\sum_{|\al|\leq
N-1}\int_{\R^3}\na_x\pa^\al a\cdot \pa^\al b dx \right\}\nonumber\\
 &&\ \ \ \ +\kappa_4\kappa_3\left\{-\sum_{1\leq |\al|\leq
N-1}\int_{\R^3}\pa^\al E\cdot \pa^\al b dx-\kappa_2\sum_{1\leq
|\al|\leq N-2}\int_{\R^3}\pa^\al \na_x\times B\cdot \pa^\al E
dx\right\}\nonumber\\
&&\ \ \ \  +\kappa_5\sum_{k=1}^NC_k\sum_{\substack{|\be|=k \nonumber\\
|\al|+|\be|\leq N}}\|\pa^\al_\be \{\FI-\FP\}u\|^2.
\end{eqnarray}
Due to smallness of $\sup_{0\leq t\leq T}\|(a+2c)(t)\|_{L^\infty}$
by the assumption, from \eqref{lem.non.p06} and \eqref{thm.efli.00},
it is easy to see
\begin{eqnarray*}
% \nonumber to remove numbering (before each equation)
 \CE_N(U(t))  &\sim &\CE_N^{{\rm lin}}(U(t))+\sum_{k=1}^NC_k\sum_{\substack{|\be|=k \\
|\al|+|\be|\leq N}}\|\pa^\al_\be \{\FI-\FP\}u\|^2,
\end{eqnarray*}
which further implies
\begin{eqnarray*}
% \nonumber to remove numbering (before each equation)
 \CE_N(U(t))
&\sim & \|u(t)\|_{L^2_\xi
    (H^N_x)}^2+\|[E(t),B(t)]\|_{H^N}^2+\sum_{k=1}^NC_k\sum_{\substack{|\be|=k \\
|\al|+|\be|\leq N}}\|\pa^\al_\be \{\FI-\FP\}u\|^2\\
&\sim &\|u(t)\|_{H^N_{x,\xi}}^2+\|[E(t),B(t)]\|_{H^N}^2.
\end{eqnarray*}
Thus, $\CE_N(U(t))$ satisfies \eqref{thm.non.1} for any
$\kappa_5>0$. Moreover, by taking $\kappa_5>0$ small enough, the
summation of \eqref{lem.non.p01}, \eqref{lem.non.p03},
$\kappa_4\times$\eqref{lem.non.p05} and then
$\kappa_5\times$\eqref{lem.non.p05} leads to \eqref{lem.non.1} with
$\CD_N(U(t))$ defined in \eqref{thm.non.2}. This completes the proof
of Lemma \ref{lem.non}.
\end{proof}

\noindent{\bf Proof of Theorem \ref{thm.non}:} Let us consider the  uniform-in-time a priori estimates
of solutions under the smallness assumption that
\begin{equation*}
    \sup_{0\leq t\leq T}\left(\|u(t)\|_{H^N_{x,\xi}}^2+\|[E(t),B(t)]\|_{H^N}^2\right)\leq \de
\end{equation*}
for a sufficiently small constant $\de>0$. This smallness assumption also implies that $\sup_{0\leq t\leq T}\|(a+2c)(t)\|_{L^\infty}$ is small enough since $N\geq 4$. Thus, it follows from Lemma \ref{lem.non} that
there are $\CE_N(U(t))$,  $\CD_N(U(t))$ defined in \eqref{lem.non.p07} and \eqref{thm.non.2} such that
\eqref{thm.non.1} holds true for $\CE_N(U(t))$ and
\begin{equation*}
    \frac{d}{dt}\CE_N(U(t))+\la \CD_N(U(t))\leq C (\de^{1/2}+\de)\CD_N(U(t))
\end{equation*}
for any $0\leq t\leq T$, that is
\begin{equation*}
\CE_N(U(t))+\la\int_0^t\CD_N(U(s))ds\leq \CE_N(U_0)
\end{equation*}
for any $0\leq t\leq T$, since $\de>0$ is small enough. Now, the rest proof follows from the standard process by combining the above  uniform-in-time a priori estimates with the local existence as well as the continuity argument as in \cite{GuoVMB} or \cite{DY-09VPB} under the assumption that $ \CE_N(U_0)$  is sufficiently small, and details are omitted for simplicity. The proof of Theorem \ref{thm.non} is complete. \qed

\medskip

 {We conclude this paper with a discussion about the large-time behavior of solutions to the nonlinear system.
Although Theorem \ref{thm.non} shows the global existence of close-to-equilibrium solutions to the Cauchy problem \eqref{VMB.eq}-\eqref{VMB.ID} of the nonlinear Vlasov-Maxwell-Boltzmann system, the decay rate of the obtained solution remains open. This issue has been studied in   \cite{DS-VMB}  for the case of two-species. However, the approach of \cite{DS-VMB} by applying the linear decay property together with the Duhamel's principle to the nonlinear system   can not be applied to the case of one-species here. Let us explain a little the key difficulty in a formal way. In fact,  the linear system in one-species decays as $(1+t)^{-3/8}$ which is slower than $(1+t)^{-3/4}$ in two-species as pointed out in Table 1. Thus, in one-species case,  the quadratic nonlinear source decays as at most $(1+t)^{-3/4}$, and if the   Duhamel's principle was used, the time-integral term generated from the nonhomogeneous source decays as}
$$
 {\int_0^t(1+t-s)^{-\frac{3}{8}} (1+s)^{-\frac{3}{4}}\,ds\leq C (1+t)^{-\frac{1}{8}}.}
$$
 {Thus, the bootstrap argument breaks down and one can not expect the solution to decay as $(1+t)^{-3/8}$ in the nonlinear case. Therefore, the study of the large-time behavior for the one-species Vlasov-Maxwell-Boltzmann system becomes much more difficult than for the two-species case as in \cite{DS-VMB}.}

\bigskip

\noindent {\it Acknowledgments.}  {This work was supported partially by the Direct Grant 2010/2011 from CUHK and by the General Research Fund (Project No.~400511) from RGC of Hong Kong.} The author would like to thank Peter Markowich and Massimo Fornasier for their strong support when he stayed in RICAM, Austrian Academy of Sciences for the Post-doc study during 2008-10, and thank Professors Tong Yang, Chang-Jiang Zhu and Hui-Jiang Zhao for their continuous encouragement. Some useful discussions from Robert M. Strain are much acknowledged.  {Thanks also go to the referee for useful suggestions to improve the presentation of the paper.}

\bigskip

%\addcontentsline{toc}{section}{References}

\end{document}